\documentclass[12pt,reqno]{amsart}
\usepackage{amssymb,latexsym,amsmath,epsfig,amsthm,mathrsfs}

\usepackage{rotating}
\usepackage{graphicx}
\usepackage{amssymb}
\usepackage{lineno}
\usepackage{enumitem}
\usepackage{cite}
\usepackage{multicol}
\usepackage[top=3cm, bottom=3cm, left=2.5cm, right=2.5cm]{geometry}
\usepackage[usenames]{color}
\usepackage[colorlinks=true,
linkcolor=blue,
filecolor=blue,
citecolor=blue]{hyperref}

\newtheorem{theorem}{Theorem}
\newtheorem{lemma}{Lemma}

\newtheorem{corollary}{Corollary}

\theoremstyle{definition}

\newtheorem{conjecture}{Conjecture}

\begin{document}
	
	\title[Restricted Signed sumset in  set of integers ]{Some direct and inverse problems for the Restricted Signed sumset in  set of integers}
	
	
	\author[Mohan]{Mohan}
	\address{Department of Mathematics, Indian Institute of Technology Roorkee, Uttarakhand, 247667, India}
	\email{mohan98math@gmail.com}

	\author[R K Mistri]{Raj Kumar Mistri}
	\address{Department of Mathematics, Indian Institute of Technology Bhilai, Durg, 491001, Chhattisgarh, India}
	\email{rkmistri@iitbhilai.ac.in}
	
	\author[R K Pandey]{Ram Krishna Pandey}
	\address{Department of Mathematics, Indian Institute of Technology Roorkee, Uttarakhand, 247667, India}
	\email{ram.pandey@ma.iitr.ac.in}
	
	\subjclass[2010]{11P70, 11B75, 11B13}
	
	
	
	\keywords{Sumset, restricted sumset, signed sumset.}

	\begin{abstract}
	Given a positive integer $h$ and a nonempty finite set of integers $A=\{a_{1},a_{2},\ldots,a_{k}\}$, the \textit{restricted $h$-fold signed sumset of $A$}, denoted by $h^{\wedge}_{\pm}A$, is defined as
	$$h^{\wedge}_{\pm}A=\left\lbrace \sum_{i=1}^{k} \lambda_{i} a_{i}: \lambda_{i} \in \left\lbrace -1, 0, 1\right\rbrace \ \text{for} \ i=  1, 2, \ldots, k \ \text{and} \  \sum_{i=1}^{k} \left| \lambda_{i} \right| =h\right\rbrace. $$ 
	The direct problem associated with this sumset is to find the optimal lower bound of $|h^{\wedge}_{\pm}A|$, and the inverse problem associated with this sumset is to determine the structure of the underlying set $A$, when $|h^{\wedge}_{\pm}A|$ attains the optimal lower bound. Bhanja, Komatsu and Pandey studied the direct and inverse  problem for the restricted $h$-fold signed sumset for $h=2, 3$, and $k$ and conjectured some direct and inverse results for $h \geq 4$. In this paper, we prove these conjectures for $h=4$. We also prove the direct and inverse theorems for arbitrary $h$ under certain restrictions on the set $A$ which are particular cases of the conjectures. Moreover, we prove these conjectures for arithmetic progressions.
	\end{abstract}
\maketitle
	\section{Introduction}	

The following notation will be use throughout the paper. Let $\Bbb Z$ denote the set of integers. For  integers $a$ and $b$ with $a \leq b$, we denote the set $\{n \in \Bbb Z: a \leq n \leq b\}$ by $[a, b]$. For a nonempty finite set $A$ of integers, let $\max(A)$, $\min(A)$, $\max_{-}(A)$, $\min_{+}(A)$ denote the largest, the smallest, the second largest and the second smallest elements of $A$, respectively. For an integer $c$, we denote the set $\{ca: a \in A\}$ by $c \ast A$, and  write $-A$ for $(-1) \ast A$. By a $k$-term arithmetic progression of integers with common difference $d$, we mean a set $A$ of the form $\{a + id: i = 0, 1, \ldots, k-1\}$. Let $a, b, u, v, u_1, u_2, \ldots, u_n$ be integers. Then we write $a < \{u_1, u_2, \ldots, u_n\} < b$ to mean  $a < u_i < b$ for all $i = 1, 2, \ldots, n$. We also write $a <\{u~\text{or}~v\} < b$ to  mean  either $a < u < b$ or $a < v < b$. A set $S$ is said to be symmetric if $x\in S$, then $-x\in S$.  Let $A = \{a_1, a_2, \ldots, a_k\}$ be a nonempty finite subset of an additive abelian group $G$. For a positive integer $h$, the classical $h$-fold sumset $hA$ and the restricted $h$-fold sumset $h^{\wedge}A$ are defined as follows:
\[hA = \left\{\sum_{i=1}^{k}\lambda_i a_i: \lambda_i \in [0, h] ~\text{for}~ i = 1, 2, \ldots, k ~\text{and}~ \sum_{i=1}^{k}\lambda_i = h\right\},\]
and
\[h^{\wedge}A =  \left\{\sum_{i=1}^{k}\lambda_i a_i: \lambda_i \in [0, 1] ~\text{for}~ i = 1, 2,\ldots, k ~\text{and}~ \sum_{i=1}^{k}\lambda_i = h\right\}.\]

Other related sumsets of two subsets $A$ and $B$ of $G$ are the Minkowski sumset $A + B := \{a + b: a \in A, b \in B\}$, and the restricted sumset $A \dotplus B := \{a + b: a \in A, b \in B, a \neq b\}$. The study of the sumsets dates back to Cauchy \cite{Cauchy} who obtained the lower bound for the cardinality of the sumset $A + B$, where $A$ and $B$ are nonempty subsets of the set of residue classes modulo a prime $p$. The result is known as {\em Cauchy-Davenport Theorem} after Davenport rediscovered this result \cite{dav1, dav2} in $1935$. These type of sumsets have been studied extensively in the literature. A  classical book by Nathanson \cite{Nathanson1996} on Additive number theory contains detailed study of these sumsets and other kind of sumsets, and has a comprehensive list of bibliography (see \cite{tao} and \cite{Bajnok2018} also). For some old and recent articles in the context of $h$-fold sumsets and restricted $h$-fold  sumsets and their generalizations, see \cite{Freiman1959,Freiman1973,Lev1996,MistriPandey2014,Monopoli2015,YangChen2015,MistryPandey2018,TangWing2019,TANG2021,MohanPandey2022}.

Two other variants of these sumsets have appeared recently in the literature \cite{bajnok-ruzsa2003,KlopschLev2003, klopsch-lev2009, Bajnok2018, BhanjaKomPandey2021}: The {\it $h$-fold signed sumset} $h_{\pm}A$ and the {\it restricted $h$-fold signed sumset} $h^{\wedge}_{\pm}A$ of the set $A$. These are defined as follows:
\begin{equation*}
	h_{\pm}A := \left\{\sum_{i=1}^{k} \lambda_i a_i : \lambda_i \in [-h, h] ~\text{for}~ i = 1, 2, \ldots ,k ~\text{and}~ \sum_{i=1}^{k} |\lambda_i| =h \right\},
\end{equation*}
and
\begin{equation*}
	h^{\wedge}_{\pm}A := \left\{\sum_{i=1}^{k} \lambda_i a_i : \lambda_i \in [-1, 1] ~\text{for}~ i = 1, 2, \ldots ,k ~\text{and}~ \sum_{i=1}^{k} |\lambda_i| =h \right\}.
\end{equation*}

The problems associated with these sumsets are \textit{direct problems} and \textit{inverse problems}. The study of optimal lower bound for the cardinality of a sumset of a given set $A$ is called a direct problem, and the study of the structure of the underlying set $A$, when the optimal lower bound for the cardinality of its sumset is known, is called an inverse problem.

The signed sumset first appeared in the work of Bajnok and Ruzsa \cite{bajnok-ruzsa2003} in the context of the independence number of a subset of an abelian group, and it also appeared in the work of Klopsch and Lev \cite{KlopschLev2003,klopsch-lev2009} in a different context. Not much is known for the signed sumset. For the  results in this contex, one may refer Bajnok and Matzke \cite{bajnok-matzke2015, bajnok-matzke2016}. In a recent paper \cite{BhanjaPandey2019}, Bhanja and Pandey have studied the direct and inverse problems in the additive group $\mathbb{Z}$. They obtained the optimal lower bound for the cardinality of the sumset $h_{\pm}A$. They also proved that if the optimal lower bound is achieved, then $A$ must be a certain arithmetic progression.

Much less is known for the restricted signed sumset $h^{\wedge}_{\pm}A$. Recently Bhanja et al. \cite{BhanjaKomPandey2021} solved some cases of both the direct and inverse problems for  $h^{\wedge}_{\pm}A$ in $\Bbb Z$ and  conjectured  rest of the cases. More precisely, they proved the following result.

\begin{theorem}[{\cite[Theorem 2.1, Theorem 3.1]{BhanjaKomPandey2021}}]
	Let $h$ and $k$ be positive integers with $h \leq k$. Let $A$ be a set of $k$ nonnegative integers. If $0 \notin A$, then
	\begin{equation}\label{Bound A}
		\left|h^{\wedge}_{\pm}A\right| \geq   2(hk-h^2)+ \frac{h(h+1)}{2} + 1.
	\end{equation}
	If $0 \in A $, then
	\begin{equation}\label{Bound B}
		\left|h^{\wedge}_{\pm}A\right| \geq   2(hk-h^2)+ \dfrac{h(h-1)}{2} + 1.
	\end{equation}
	These lower bounds are best possible for $h=1,2$, and k.
\end{theorem}

In the same paper, they also proved the inverse theorems for $h=2$ and $h=k$ (see \cite[Theorem 2.2, Theorem 2.3, Theorem 3.2, Theorem 3.3]{BhanjaKomPandey2021}). It can be verified that the lower bounds in (\ref{Bound A}) and (\ref{Bound B}) are not optimal for $3 \leq h \leq k-1$. For these cases, they conjectured the lower bounds and the inverse results, and proved their conjectures for the case $h = 3$ (see \cite[Theorem 2.5]{BhanjaKomPandey2021} and \cite[Theorem 3.5]{BhanjaKomPandey2021})). The precise statements of the conjectures are the following.

\begin{conjecture}[{\cite[Conjecture 2.4, Conjecture 2.6]{BhanjaKomPandey2021}}]\label{Conjecture 1}			
	Let $A$ be a set of $k \geq 4$ positive integers, and let $h$ be an integer with $3 \leq h \leq k-1$. Then
	\begin{equation}\label{Lower bound Conjecture 1}
		\left|h^{\wedge}_{\pm}A\right| \geq 2hk-h^2 + 1.
	\end{equation}
	This lower bound is best possible. Moreover,  if $|h^{\wedge}_{\pm}A| = 2hk-h^2 + 1,$ then $A = d \ast \{1,3,\ldots, 2k-1\}$ for some positive integer $d$.
\end{conjecture}

\begin{conjecture}[{\cite[Conjecture 3.4, Conjecture 3.7]{BhanjaKomPandey2021}}]\label{Conjecture 2}
	Let $A$ be a set of $k \geq 5$ nonnegative integers with $0 \in A$, and let $h$ be an integer with $3 \leq h \leq k-1$. Then
	\begin{equation*}
		\left|h^{\wedge}_{\pm}A\right| \geq 2hk - h(h+1) + 1.
	\end{equation*}
	This lower bound is best possible. Moreover, if $\left|h^{\wedge}_{\pm}A\right| = 2hk-h(h+1) + 1,$ then $A = d \ast [0,k-1]$ for some positive integer $d$.
\end{conjecture}
We prove these conjectures for the case $h =4$ in Section \ref{proof-conj1}. In  Section \ref{sec-aux-lem}, we prove some auxiliary lemmas which will be crucial in the proof of the conjectures for $h = 4$. Using these lemmas, we also prove the conjectures for certain special type of sets including arithmetic progression (see the results in Subsection \ref{results-special-sets}). We remark that Bhanja et al. \cite{BhanjaKomPandey2021}  proved the lower bound in (\ref{Lower bound Conjecture 1}) for super increasing sequences.

\section{Two auxiliary lemmas and some special cases} \label{sec-aux-lem}

\subsection{Auxiliary lemmas}

We need the following result for the proof of Lemma \ref{Lemma-1} and Lemma \ref{Lemma-2}.

\begin{theorem} [{\cite[Theorem 1.9, Theorem 1.10]{Nathanson1996}}]  \label{Nathanson Direct Inverse  Thm}
	Let A be a nonempty finite set of integers, and let $1 \leq h \leq \left| A\right| $. Then
	\[\left|h^{\wedge}A\right| \geq h|A| - h^{2} + 1.\]
	Furthermore, if $\left|h^{\wedge}A\right| = h|A| - h^{2} + 1$, then $A$ is an arithmetic progression provided that $\left|A\right| \geq 5$ and $2 \leq h \leq |A|-2$.
\end{theorem}

\begin{lemma}\label{Lemma-1}
	Let $h$ and $k$ be integers with $3 \leq h \leq k-1$. Let $A = \{a_{1}, a_{2}, \ldots, a_{k}\}$ be a finite set of $k$ positive integers, where $a_1 < a_2 < \cdots < a_k$. Let $A_{h+1} = \{a_{1}, a_{2}, \ldots, a_{h+1}\} \subseteq A$. If $\left|h^{\wedge}_{\pm}A_{h+1}\right| \geq (h+1)^2 + t$, where $t \geq 0$,  then
	\begin{equation*}
		\left|h^{\wedge}_{\pm}A\right| \geq 2hk - h^2 + 1 + t.
	\end{equation*}
\end{lemma}

\begin{proof}
	Let $A^{\prime} = \{a_{2}, a_{3},\ldots,a_{k}\}$. Then $(-h^{\wedge}A^{\prime}) \cup h^{\wedge}_{\pm}A_{h+1} \cup h^{\wedge}A^{\prime} \subseteq h^{\wedge}_{\pm}A$. Since
	\[h^{\wedge}_{\pm}A_{h+1} \cap h^{\wedge}A^{\prime} = \{a_{2} + a_{3} +  \cdots+a_{h+1}\}\]
	and
	\[h^{\wedge}_{\pm}A_{h+1} \cap (-h^{\wedge}A^{\prime}) = \{-(a_{2} + a_{3} +  \cdots+a_{h+1})\},\]
	it follows from Theorem \ref{Nathanson Direct Inverse  Thm} that
	\begin{equation*}
		\left| h^{\wedge}_{\pm}A\right| \geq \left| h^{\wedge}_{\pm}A_{h+1}\right| + 2\left|  h^{\wedge}A^{\prime}\right|-2  \geq (h+1)^{2} +t +2 (h(k-1)- h^{2}+1) -2 = 2hk - h^2 + 1 + t.
	\end{equation*}
	This proves the lemma.
\end{proof}

A similar argument proves the following lemma.

\begin{lemma}\label{Lemma-2}		
	Let $h$ and $k$ be integers with $k \geq 5$ and $3 \leq h \leq k-1$. Let $A = \{a_{0}, a_{1}, a_{2}, \ldots, a_{k-1}\}$ be a finite set of $k$ nonnegative integers, where $0 = a_{0} < a_1 < a_2 < \cdots < a_{k-1}$. Let $A_{h} = \{a_{0}, a_{1}, a_{2}, \ldots, a_{h}\} \subseteq A$. If $\left|h^{\wedge}_{\pm}A_{h}\right| \geq h(h+1) + 1+ t$, where $t \geq 0$, then
	\begin{equation*}
		\left|h^{\wedge}_{\pm}A\right| \geq 2hk-h(h+1) + 1 +t.
	\end{equation*}
\end{lemma}

\subsection{Some special cases}\label{results-special-sets}

First we verify the conjectures for finite arithmetic progressions.

\begin{theorem}\label{thm 1}
	Let $h \geq 3$ be an integer, and let  $A$ be a $(h+1)$-term arithmetic progression of positive integers with common difference $d$, then
	\[
	\left|h^{\wedge}_{\pm}A\right| \geq
	\begin{cases}
		(h + 1)^2, & \mbox{if } d = 2 \min(A); \\
		(h + 1)^2 + 1, & \mbox{otherwise}.
	\end{cases}
	\]
	Furthermore, $\left|h^{\wedge}_{\pm}A\right| = (h+1)^{2}$ if and only if $d = 2\min(A)$.
	
\end{theorem}

\begin{proof}
	Let $A = \{a_{1}, a_{2}, \ldots, a_{h+1}\}$, where $a_{i} = a_{1} + (i-1)d$ for $i = 1, 2, \ldots, h+1$. Set  $A_{h} = A \setminus\{a_{h+1}\}$. We shall use induction on $h$ to prove the result. If $h=3$, then $A = \{a_{1}, a_{2}, a_{3}, a_{4}\}$ with $a_{2} - a_{1} = a_{3} - a_{2} = a_{4} - a_{3} = d$. Consider the following increasing sequence of the  elements of $3^{\wedge}_{\pm}A$.
	\begin{multline*}
		-a_{4} - a_{3} - a_{2} < -a_{4} - a_{3} - a_{1} < -a_{4} - a_{2} - a_{1} < \{ -a_{3} - a_{2} - a_{1} \ \text{or}  \ -a_{4} - a_{2} + a_{1} \} \\
		<  -a_{3} - a_{2} + a_{1} < -a_{4} + a_{2} - a_{1} < -a_{3} + a_{2} - a_{1} < -a_{2} + a_{4} - a_{3} < -a_{1} + a_{3} - a_{2} \\
		< a_{1} -a_{2} +a_{3} < a_{1} - a_{2} +  a_{4} < -a_{1} + a_{2} + a_{3} < a_{1} + a_{2} + a_{3} < a_{1} + a_{2} + a_{4}\\
		< a_{1} + a_{3} + a_{4} < a_{2} + a_{3} + a_{4}.
	\end{multline*}
	Since $d = 2a_{1}$ if and only if  $-a_{3} - a_{2} - a_{1} = -a_{4} - a_{2} + a_{1}$, it follows that
	\[
	\left|3^{\wedge}_{\pm}A\right|\geq
	\begin{cases}
		16, & \mbox{if } d=2a_{1};\\
		17, & \mbox{if } d\neq 2a_{1}.
	\end{cases}
	\]
	Now assume that the result holds for  $h-1 \geq 3$. Then
	\[\left|(h-1)^{\wedge}_{\pm}A_{h}\right| \geq
	\begin{cases}
		h^{2}, & \mbox{if }  d = 2a_{1};\\
		h^{2} + 1, & \mbox{if } d \neq 2a_{1}.
	\end{cases}
	\]
	Since $(h-1)^{\wedge}_{\pm}A_{h} + a_{h+1} \subseteq h^{\wedge}_{\pm}A$ and $|(h-1)^{\wedge}_{\pm}A_{h} + a_{h+1}| = |(h-1)^{\wedge}_{\pm}A_{h}|$, we have at least $h^2$ elements in $h^{\wedge}_{\pm}A$ if $d = 2a_1$, and at least $h^2 + 1$ elements in $h^{\wedge}_{\pm}A$ if $d \neq 2a_1$.   Now, we construct some more elements of $h^{\wedge}_{\pm}A$ distinct form the elements of $(h-1)^{\wedge}_{\pm}A_{h} + a_{h+1}$.  For each $i \in [3, h+1]$, let	
	\[S_{i} = -a_{2} + a_{i} -\left( \sum_{j = 3, j \neq i}^{h+1}a_{j} \right).\]
	For each $i \in [3, h]$, let
	\[T_{i} = -a_{1}  + a_{i}  - \left(\sum_{j = 3, j \neq i}^{h+1} a_{j}\right).\]
	Let
	\[U_{0} = -\sum_{j=2}^{h+1}a_{j} \ \text{and} \ U_{1} = -\sum_{j=1, j\neq 2}^{h+1}a_{j},\]
	and for each $i \in [2,5]$, let
	\[U_{i} = a_{1} - \left(\sum_{j = 2, j \neq i}^{h+1}a_{j} \right). \]
	It is easy to see that
	\begin{multline*}
		\min ((h-1)^{\wedge}_{\pm}A_{h} + a_{h+1} ) = S_{h+1} > T_{h} > S_{h} > T_{h-1} > S_{h-1} > \cdots > T_{3} \\
		>S_{3} = U_{5} > U_{4} > U_{3} > U_{2} > U_{1} > U_{0} = \min(h^{\wedge}_{\pm}A).
	\end{multline*}
	Hence
	\[
	\left|h^{\wedge}_{\pm}A\right| \geq
	\begin{cases}
		(h+1)^{2}, & \mbox{if } d=2a_{1};\\
		(h+1)^{2} + 1, & \mbox{if } d\neq 2a_{1}.
	\end{cases}
	\]
	If $d = 2a_{1}$, then $A = a_{1} \ast \{1, 3, 5, \ldots,  2h+1\}$. Therefore, if $h$ is an odd positive integer, then  $h^{\wedge}_{\pm}A$ contains only odd multiples of $a_{1}$  and  if $h$ is an even positive integer, then $h^{\wedge}_{\pm}A$ contains only  even multiples of $a_{1}$. Since $h^{\wedge}_{\pm}A \subseteq a_{1} \ast [-(h+1)^2 +1, (h+1)^2 -1]$, we get  $\left| h^{\wedge}_{\pm}A \right| \leq (h+1)^2$. This completes the proof of the theorem.
\end{proof}

\begin{corollary}\label{Coro-1}
	Let $h \geq 3$ be an integer. Let $A = \{a_{1}, a_{2}, \ldots, a_{k}\}$ be an arithmetic progression of length $k \geq h+1$ integers with common difference $d$, where $0<a_1 < a_2 < \cdots < a_k$. Then
	\[
	\left|h^{\wedge}_{\pm}A\right| \geq
	\begin{cases}
		2hk - h^{2} +1, & \mbox{if } d=2a_{1};\\
		2hk - h^{2} + 2, & otherwise.
	\end{cases}
	\]
	Furthermore, $\left|h^{\wedge}_{\pm}A\right| = 2hk - h^{2} +1$ if and only if $d = 2\min(A)$.
\end{corollary}

\begin{proof}
	Let $A_{h+1}= \{a_{1}, a_{2}, \ldots, a_{h+1}\}$ and $A^{\prime} = \{a_{2}, a_{3}, \ldots, a_{k}\}$. Then $(-h^{\wedge}A^{\prime}) \cup h^{\wedge}_{\pm}A_{h+1} \cup h^{\wedge}A^{\prime} \subseteq h^{\wedge}_{\pm}A$. Note that
	\[h^{\wedge}_{\pm}A_{h+1} \cap h^{\wedge}A^{\prime} = \{a_{2} + a_{3} +  \cdots+a_{h+1}\}\]
	and
	\[h^{\wedge}_{\pm}A_{h+1} \cap (-h^{\wedge}A^{\prime}) = \{-(a_{2} + a_{3} +  \cdots+a_{h+1})\}.\] 
	Using Theorem \ref{thm 1} and  Theorem \ref{Nathanson Direct Inverse  Thm}, we have
	\begin{align}\label{coro-eq}
		\left| h^{\wedge}_{\pm}A\right| &\geq \left| h^{\wedge}_{\pm}A_{h+1}\right| + 2\left|  h^{\wedge}A^{\prime}\right|-2 \\
		&\geq 	\begin{cases}
			(h+1)^{2} + 2 (h(k-1)- h^{2}+1) -2, & \mbox{if } d=2a_{1};\\
			(h+1)^{2} + 1 + 2 (h(k-1)- h^{2}+1) -2, & otherwise.
		\end{cases} \nonumber\\ 
		&=
		\begin{cases}
			2hk - h^{2} +1, & \mbox{if } d=2a_{1};\\
			2hk - h^{2} + 2, & otherwise.
		\end{cases}\nonumber
	\end{align}
	Now, if $d=2a_{1}$, then $A = a_{1} \ast \{1, 3, 5, \ldots,  2k-1\}$ and so $\left|h^{\wedge}_{\pm}A\right| = 2hk - h^{2} +1$.
	If $\left|h^{\wedge}_{\pm}A\right| = 2hk - h^{2} +1$ , then by (\ref{coro-eq}), we get $\left| h^{\wedge}_{\pm}A_{h+1}\right| = (h+1)^{2}$ and so $d=2a_{1}$.
\end{proof}

\begin{theorem}
	Let $h$ and $k$ be integers such that $4 \leq h\leq k-1$.	If $A$ is an arithmetic progression of $k$ nonnegative integers with $0 \in A$, then $|h^{\wedge}_{\pm}A| = 2hk-h(h+1) + 1$.
\end{theorem}

\begin{proof}
	If $A$ is an arithmetic progression of $k$ nonnegative integers with $0 \in A$, then $A = d\ast [0,k-1]$, where $d$ is the common difference of the arithmetic progression. Since cardinality of  restricted $h$-fold signed sumset is dilation invariant, we may assume $A = [0,k-1]$. Therefore $h^{\wedge}_{\pm}A$ contains disjoint sets  $h^{\wedge}A$ and  $h^{\wedge}(-A)$. It is easy to see that $h^{\wedge}A = \left[\dfrac{h(h-1)}{2}, hk-\dfrac{h(h+1)}{2}\right]$ and    $h^{\wedge}(-A) = - (h^{\wedge}A )$.  Now we construct some more elements of  $h^{\wedge}_{\pm}A$ different from the elements of  $h^{\wedge}A $ and  $h^{\wedge}(-A)$. It is easy to see that
	\begin{align*}
		\max (h^{\wedge}(-A)) &= 0-1-2- \cdots -(h-2)-(h-1)\\
		&<0+1-2- \cdots-(h-2)-h\\
		&<0+1-2- \cdots-(h-2)-(h-1)\\
		&<0-1+2- \cdots-(h-2)-h\\ 
		&<0-1+2- \cdots-(h-2)-(h-1)\\
		&\vdots\\ 
		&<0-1- \cdots-(h-3)+(h-2)-(h-1)\\
		&<0+1- \cdots-(h-3)+(h-2)-h\\
		&<0+1- \cdots-(h-3)+(h-2)-(h-1)\\
		&<0-1+2 \cdots-(h-3)+(h-2)-h\\
		&<0-1+2 \cdots-(h-3)+(h-2)-(h-1)\\
		&\vdots\\
		&<0-1 \cdots +(h-3)+(h-2)-h\\
		&\vdots\\
		&<0+1+2+ \cdots + (h-2) -(h-1) \\
		&=  0-1+2+ \cdots+(h-3) - (h-2) +(h-1)\\
		&<0-1+2+ \cdots+(h-3) - (h-2) + h\\
		&<0-1+2+ \cdots-(h-3) + (h-2) +(h-1)\\
		&\vdots	\\
		&<0-1-2 + \cdots + (h-2) + h\\
		&<0+1-2 + \cdots + (h-2) + (h-1)\\
		&<0+1-2 + \cdots + (h-2) + h\\
				\end{align*}
	\begin{align*}
		&<0-1+2 + \cdots + (h-2) + (h-1)\\
		&<0-1+2 + \cdots + (h-2) + h <0+1+2 + \cdots + (h-2) + (h-1) = \min(h^{\wedge}A).
	\end{align*}
	\noindent		Hence $h^{\wedge}_{\pm}A = \left[-hk+\dfrac{h(h+1)}{2}, hk-\dfrac{h(h+1)}{2} \right]$ and $|h^{\wedge}_{\pm}A| = 2hk-h(h+1) + 1$.
\end{proof}

Next theorem is a partial inverse theorem.

\begin{theorem}\label{Partial inverse thm}
	Let $h$ and $k$ be integers with $4 \leq h \leq k-1$. Let $A = \{a_{1}, a_{2}, \ldots, a_{k}\}$ be a set of $k$ positive integers such that $a_{1} < a_{2} < \cdots < a_{k}$. Let  $A_{h+1}  = \{a_{1}, a_{2}, \ldots, a_{h+1}\}$ and $A^{\prime} = A \setminus \{a_{1}\}$. Suppose that
	\begin{equation}\label{Inverse conjecture eq-1 }
		|h^{\wedge}_{\pm}A| = 2hk-h^2 + 1,
	\end{equation}
	and one of the following conditions holds:
	\begin{enumerate}
		\item [\upshape(a)] $A$ is an arithmetic progression.
		\item [\upshape(b)]  $A_{h+1}$ is an arithmetic progression
		\item [\upshape(c)] $\left| h^{\wedge}_{\pm}A_{h+1}\right|   \geq  (h+1)^2 $ and $4 \leq h \leq k-3$.
		\item [\upshape(d)] $ h^{\wedge}_{\pm}A =  h^{\wedge}(-A^{\prime}) \cup  h^{\wedge}_{\pm}A_{h+1} \cup  h^{\wedge}A^{\prime}$ and $A^{\prime}$ is an arithmetic progression.
		\item [\upshape(e)] $\left| h^{\wedge}_{\pm}A_{h+1}\right| \geq  (h+1)^2$ and $A^{\prime}$ is an arithmetic progression.
	\end{enumerate}
	Then $A = a_1 \ast \{1,3,5, \ldots, 2k-1\}$.
\end{theorem}

\begin{proof}
	We have $$(-h^{\wedge}A^{\prime}) \cup h^{\wedge}_{\pm}A_{h+1} \cup h^{\wedge}A^{\prime} \subseteq h^{\wedge}_{\pm}A.$$
	\begin{enumerate}
		\item [\upshape(a)] If $A$ is an arithmetic progression, then  Corollary \ref{Coro-1} implies that $A = a_1 \ast \{1,3,5, \ldots, 2k-1\}$. 
		\item [\upshape(b)] If $A_{h+1}$ is an arithmetic progression, then  Theorem \ref{thm 1} implies that $\left| h^{\wedge}_{\pm}A_{h+1}\right|   \geq  (h+1)^2$. Now, if $k = h+1$, then $|h^{\wedge}_{\pm}A_{h+1}|=|h^{\wedge}_{\pm}A| = (h+1)^2,$ and  using Theorem \ref{thm 1}, we have $A_{h+1} = a_{1} \ast \{1,3,5, \ldots, 2h+1\}$. Assume $k \geq h+2$. Note that $(-h^{\wedge}A^{\prime}) \cup h^{\wedge}_{\pm}A_{h+1} \cup h^{\wedge}A^{\prime} \subseteq h^{\wedge}_{\pm}A$ and there is exactly one element in each of  $h^{\wedge}_{\pm}A_{h} \cap h^{\wedge}A^{\prime}$ and $h^{\wedge}_{\pm}A_{h} \cap (- h^{\wedge}A^{\prime})$. Therefore, by Theorem \ref{Nathanson Direct Inverse  Thm}, we have
		\begin{align*}
			2hk-h^{2} + 1 &= \left| h^{\wedge}_{\pm}A\right| \\
			& \geq \left| h^{\wedge}_{\pm}A_{h+1}\right| + \left|  h^{\wedge}A^{\prime}\right|  + \left|  h^{\wedge}(-A^{\prime})\right|- 2 \\
			& \geq (h+1)^2 + 2h(k-1)-2h^{2}+2-2= 2hk-h^{2} + 1
		\end{align*}
		This gives
		\begin{equation}\label{Eq-1}
			\left| h^{\wedge}_{\pm}A_{h+1}\right|   = (h+1)^2,
		\end{equation}
		which, by Theorem \ref{thm 1}, gives $A_{h+1} =a_{1}\ast \{1,3,5,\ldots, 2h+1\}$. Also, we have 
		\begin{equation}\label{Eq-2}
			\left|  h^{\wedge}A^{\prime}\right|  + \left|  h^{\wedge}(-A^{\prime})\right| = 2h(k-1)-2h^{2}+2
		\end{equation}	
		and
		\begin{equation}\label{Eq-3}
			h^{\wedge}_{\pm}A =  h^{\wedge}(-A^{\prime}) \cup  h^{\wedge}_{\pm}A_{h+1} \cup  h^{\wedge}A^{\prime}.
		\end{equation}
		Let $x =  a_{1}+a_{3} + \cdots + a_{h} + a_{h+2} \in h^{\wedge}_{\pm}A$. Note that
		\begin{multline*}
			\max\nolimits_{-}( h^{\wedge}_{\pm}A_{h+1}) = a_{1} + a_{3} + \cdots + a_{h} + a_{h+1} < x  \\< a_{2} + a_{3} + \cdots + a_{h} + a_{h+2} =  \min\nolimits_{+}(h^{\wedge}A^{\prime}) 
		\end{multline*}
		and
		\[\max\nolimits_{-}( h^{\wedge}_{\pm}A_{h+1}) < \max( h^{\wedge}_{\pm}A_{h+1}) = \min(h^{\wedge}A^{\prime}) <  \min\nolimits_{+}(h^{\wedge}A^{\prime}).\]
		From (\ref{Eq-1}), (\ref{Eq-2}) and (\ref{Eq-3}), it follows that
		\[x = a_{1}+a_{3} + \cdots + a_{h} + a_{h+2} = a_{2}+a_{3} + \cdots + a_{h} + a_{h+1} = \max( h^{\wedge}_{\pm}A_{h+1}). \]
		This gives $a_{h+2}- a_{h+1} = a_{2} - a_{1}$. Therefore, if $k=h+2$, then $A = a_{1} \ast \{1,3,5,\ldots, 2h+3\}$.  If $4 \leq h \leq k-3$, then from (\ref{Eq-2}) and Theorem \ref{Nathanson Direct Inverse  Thm}, $A^{\prime}$ is an arithmetic progression. Hence $A = a_{1} \ast \{1,3,5,\ldots,2k-1\}$.
		
		\item [\upshape(c)] If $\left| h^{\wedge}_{\pm}A_{h+1}\right|   \geq  (h+1)^2 $ and $4 \leq h \leq k-3$, then by similar argument as in $(b)$, we get $a_{h+2}- a_{h+1} = a_{2} - a_{1}$ and  $	\left|  h^{\wedge}A^{\prime}\right|  = h(k-1)-h^{2}+1$. Therefore, Theorem \ref{Nathanson Direct Inverse  Thm} implies that $A^{\prime}$ is an arithemetic progression and so $A$ is an arithmetic progression. Since	$|h^{\wedge}_{\pm}A| = 2hk-h^2 + 1$, then by Corollary \ref{Coro-1}, we get $A = a_{1} \ast \{1,3,5,\ldots,2k-1\}$.
		
		\item [\upshape(d)] If $ h^{\wedge}_{\pm}A =  h^{\wedge}(-A^{\prime}) \cup  h^{\wedge}_{\pm}A_{h+1} \cup  h^{\wedge}A^{\prime}$ and $A^{\prime}$ is an arithmetic progression, then  $	\left|  h^{\wedge}A^{\prime}\right|  = h(k-1)-h^{2}+1$ and 
		\begin{align*}
			2hk-h^{2} + 1 &= \left| h^{\wedge}_{\pm}A\right| \\
			& = \left| h^{\wedge}_{\pm}A_{h+1}\right| + \left|  h^{\wedge}A^{\prime}\right|  + \left|  h^{\wedge}(-A^{\prime})\right|- 2 \\
			& = \left| h^{\wedge}_{\pm}A_{h+1}\right| + 2h(k-1)-2h^{2}+2-2.
		\end{align*}
		Therefore $\left| h^{\wedge}_{\pm}A_{h+1}\right| = (h+1)^2 $. By similar argument as in $(b)$, we get  $a_{h+2}- a_{h+1} = a_{2} - a_{1}$, and so $A$ is an arithmetic progression. Since	$|h^{\wedge}_{\pm}A| = 2hk-h^2 + 1$, then by Corollary \ref{Coro-1}, we get $A = a_{1} \ast \{1,3,5,\ldots,2k-1\}$. 
		
		\item [\upshape(e)] This case is similar to case $(d)$.
	\end{enumerate}
\end{proof}
A similar argument proves the following theorem.	
\begin{theorem}\label{Partial inverse thm I}
	Let $h$ and $k$ be integers with $4 \leq h \leq k-1$. Let $A = \{a_{0}, a_{1}, a_{2}, \ldots, a_{k-1}\}$ be a set of $k$ nonnegative integers such that $0=a_{0} < a_{1} < a_{2} < \cdots < a_{k-1}$. Let  $A_{h}  = \{a_{0}, a_{1}, a_{2}, \ldots, a_{h}\}$ and $A^{\prime} = A \setminus \{a_{0}\}$. If
	\begin{equation}\label{Inverse conjecture eq-1 }
		|h^{\wedge}_{\pm}A| = 2hk-h(h+1) + 1
	\end{equation}
	and any one of the following conditions,
	\begin{enumerate}
		\item [\upshape(a)] $A$ is an arithmetic progression
		\item [\upshape(b)]  $A_{h}$ is an arithmetic progression
		\item [\upshape(c)] $\left| h^{\wedge}_{\pm}A_{h}\right|   \geq  h(h+1) + 1 $ and $4 \leq h \leq k-3$
		\item [\upshape(d)] $ h^{\wedge}_{\pm}A =  h^{\wedge}(-A^{\prime}) \cup  h^{\wedge}_{\pm}A_{h} \cup  h^{\wedge}A^{\prime}$ and $A^{\prime}$ is an arithmetic progression
		\item [\upshape(e)] $\left| h^{\wedge}_{\pm}A_{h}\right|   \geq  h(h+1)+1 $   and $A^{\prime}$ is an arithmetic progression
	\end{enumerate}
	holds, then $A = a_1 \ast [0,k-1]$.
\end{theorem}

\begin{theorem}\label{direct theorem 2}
	Let $h \geq 3$ be an integer. Let $A=\{a_{1}, a_{2},  \ldots,  a_{h+1} \}$ be a set of positive integers, where $a_1 < a_2 < \cdots < a_{h+1}$ and $a_{i} \geq a_{i-1} + a_{i-2}$ for $i=4, \ldots , h+1$. Then
	\begin{equation}\label{theorem bound 2}
		\left|h^{\wedge}_{\pm}A\right| \geq (h+1)^2 + 1.
	\end{equation}
\end{theorem}

\begin{proof}
	We shall use induction on  $h$ to prove the lower bound.  The base case $h=3$ is proved by Bhanja et al. (see the proof of Theorem 2.5 in \cite{BhanjaKomPandey2021}). Let $h\geq 4$ and assume the result holds for $h-1$. Let  $A_{h} =\{a_{1}, a_{2}, \ldots, a_{h}\}$. Then the induction hypothesis implies that
	\begin{equation*}
		\left|(h-1)^{\wedge}_{\pm}A_{h}\right| \geq h^2 + 1.
	\end{equation*}
	Since $ (h-1)^{\wedge}_{\pm}A_{h} + a_{h+1} \subseteq h^{\wedge}_{\pm}A$, it is sufficient to construct $2h+1$ more element to complete the proof. Since $ - a_{h-1} - a_{h} + a_{h+1} \geq a_{h-1} + a_{h} - a_{h+1}$, we have
	\begin{multline*}
		\min ((h-1)^{\wedge}_{\pm}A_{h} + a_{h+1})=-a_{2} - a_{3} - \ldots - a_{h-1} - a_{h} + a_{h+1}\\ \geq - a_{2} \ldots - a_{h-2} + a_{h-1} + a_{h} - a_{h+1}.
	\end{multline*}
	Consider the following elements of $h^{\wedge}_{\pm}A$:
	\[S_{i} = -\left(\sum_{j=2, j \neq i}^{h-1}a_{j}\right) + a_{i}  + a_{h} - a_{h+1} \ \text{for} \ i= 2, 3, \ldots, h-1; \]
	\[T_{i} = -\left(\sum_{j=2, j\neq i}^{h+1}a_{j}\right) + a_{i} \ \text{for} \ i= 2, \ldots, h-1;\]
	\[X_{1} = a_{1} - \left(\sum_{i=3, i\neq h}^{h+1}a_{i}\right) + a_{h},  \ X = a_{1} - \left(\sum_{i=3}^{h+1}a_{i}\right);\]
	\[Y_{1} = - a_{1} - \left(\sum_{i=3, i\neq h}^{h+1}a_{i}\right) + a_{h}, \ Y = - a_{1} - \left(\sum_{i=3}^{h+1}a_{i}\right);\]
	\[Z_{1} = - \left(\sum_{i=2, i\neq h}^{h+1}a_{i}\right) + a_{h}, \  \text{and} \ Z = - \left(\sum_{i=2}^{h+1}a_{i}\right).\]
	It is easy to see that
	\begin{multline*}
		\min((h-1)^{\wedge}_{\pm}A_{h} + a_{h+1}) \geq S_{h-1} > S_{h-2} > \cdots > S_{2} > X_{1} > Y_{1} > Z_{1} \\
		> T_{h-1} > T_{h-2} > \cdots > T_{2} > X > Y > Z = \min(h^{\wedge}_{\pm}A).
	\end{multline*}
	Hence
	\begin{equation*}
		|h^{\wedge}_{\pm}A| \geq (h+1)^2 + 1.
	\end{equation*}
\end{proof}

\begin{theorem}\label{direct theorem 3}
	Let $h \geq 3$ be an integer. Let $A = \{a_{1}, a_{2}, \ldots,  a_{h+1}\}$ be a set of $h+1$ positive integers, where $a_1 < a_2 < \cdots < a_{h+1}$ with $a_{3} - a_{2} < 2a_{1}$ and  $a_{i} - a_{i-1} > \dfrac{a_{2} - a_{1}}{2}$ for $i = 4, \ldots , h+1$. Then
	\begin{equation*}
		|h^{\wedge}_{\pm}A| \geq (h+1)^2 + 1.
	\end{equation*}		
\end{theorem}

\begin{proof}
	We shall use induction on $h$ to prove the lower bound. The base case $h = 3$ is proved by Bhanja et al. (see the proof of Theorem 2.5 in \cite{BhanjaKomPandey2021}). Let $h \geq 4$ and assume the result holds for $h-1$. Let $A_{h} =\{a_{1}, a_{2}, \ldots, a_{h}\} \subseteq A$. Then the induction hypothesis implies that
	\begin{equation*}
		\left|(h-1)^{\wedge}_{\pm}A_{h} \right| \geq h^2 + 1.
	\end{equation*}
	Since $(h-1)^{\wedge}_{\pm}A_{h} + a_{h+1} \subseteq h^{\wedge}_{\pm}A$, it is sufficient to construct $2h+1$ more element to complete the proof. Let
	\[S_{1} = a_{1} - \left(\sum_{j = 3}^{h+1}a_{j}\right).\]
	For each $i \in [2, h+1]$, let
	\[S_{i} =   a_{i} -\left( \sum_{j = 2, j \neq i}^{h+1}a_{i} \right).\]
	For $i \in [3, h]$, let
	\[T_{i} = -a_{1}  + a_{i}  - \left(\sum_{j = 3, j \neq i}^{h+1} a_{i}\right).\]
	For each $i \in [1, 3]$, let
	\[U_{i} = -(\sum_{j = 1, j\neq i}^{h+1}a_{j}).\]
	It is easy to see that
	\begin{multline*}
		\min ((h-1)^{\wedge}_{\pm}A + a_{h+1} ) = S_{h+1} > T_{h} > S_{h} > T_{h-1} > S_{h-1} > \cdots > T_{3}\\
		> S_{3} > S_{2} > S_{1} > U_{3} > U_{2} > U_{1} = \min(h^{\wedge}_{\pm}A).
	\end{multline*}
	Hence
	\begin{equation*}
		|h^{\wedge}_{\pm}A| \geq (h+1)^2 + 1.
	\end{equation*}
\end{proof}
\section{Proof of Conjecture \ref{Conjecture 1} and Conjecture \ref{Conjecture 2} for $h = 4$}\label{proof-conj1}
In this section, we prove the following theorem which is a special case of Conjecture \ref{Conjecture 1} (the case $h = 4$).
\begin{theorem}\label{main theorem I}
	Let $A$ be a set of $k \geq 5$ positive integers. Then $$\left|4_{\pm}^{\wedge}A \right| \geq 8k-15.$$
	Furthermore, if  $\left|4_{\pm}^{\wedge}A \right| = 8k-15$, then $A = \min(A) \ast \{1,3,\ldots, 2k-1\}$.
\end{theorem}

In view of Lemma \ref{Lemma-1} and Theorem \ref{Partial inverse thm}, it suffices to prove the following theorem which implies the above theorem.

\begin{theorem}\label{main theorem}
	Let $A$ be a set of positive integers with $|A| = 5$. Then
	$$\left|4_{\pm}^{\wedge}A \right| \geq 25.$$
	Furthermore, if  $\left|4_{\pm}^{\wedge}A \right| = 25$, then $A = \min(A) \ast \{1,3,5,7,9\}$.
\end{theorem}

For the proof, we consider various cases as lemmas. Throughout this section, the following list of (not necessarily distinct) elements of  $4_{\pm}^{\wedge}A$ will be used in the proofs:
\begin{multicols}{2}
	\begin{enumerate}
		\item[] $x_{1} = -a_{1} + a_{2} - a_{3} + a_{4}$,
		\item[] $x_{2} = a_{1} + a_{2} - a_{3} + a_{4}$,
		\item[] $x_{3} = a_{1} - a_{2} + a_{3} + a_{4}$,
		\item[] $x_{4} = -a_{1} + a_{2} + a_{3} + a_{4}$,
		\item[] $x_{5} = a_{1} + a_{2} + a_{3} + a_{4}$,
		\item[] $x_{6} = a_{1} + a_{2} + a_{3} + a_{5}$,
		\item[] $x_{7} = a_{1} + a_{2} + a_{4} + a_{5}$,
		\item[] $x_{8} = a_{1} + a_{3} + a_{4} + a_{5}$,
		\item[] $x_{9} = a_{2} + a_{3} + a_{4} + a_{5}$,
		\item[] $y_{1} = -a_{1} + a_{2} + a_{4} - a_{5}$,
		\item[] $y_{2} = -a_{1} + a_{2} - a_{4} + a_{5}$,
		\item[] $y_{3} = -a_{1} - a_{3} + a_{4} + a_{5}$,
		\item[] $y_{4} = -a_{1} - a_{2} + a_{4} + a_{5}$,
		\item[] $y_{5} = a_{1} - a_{2} + a_{4} + a_{5}$,
		\item[] $y_{6} = -a_{1} + a_{2} + a_{4} + a_{5}$,
		\item[] $y_{7} = -a_{1} + a_{3} + a_{4} + a_{5}$,
		\item[] $z_{1} = a_{1} + a_{2} + a_{4} - a_{5}$,
		\item[] $z_{2} = a_{1} + a_{2} - a_{4} + a_{5}$,
		\item[] $z_{3} = a_{1} + a_{2} - a_{3} + a_{5}$,
		\item[] $z_{4} = a_{1} - a_{2} + a_{3} + a_{5}$,
		\item[] $z_{5} = -a_{1} + a_{2} + a_{3} + a_{5}$,
		
		\item[] $\alpha_{1} = -a_{2} + a_{3} - a_{4} + a_{5}$,
		\item[] $\alpha_{2} = -a_{2} - a_{3} + a_{4} + a_{5}$,
		\item[] $\alpha_{3} = -a_{2} + a_{3} + a_{4} + a_{5}$,
		
		\item[] $\beta_{1} = -a_{1} - a_{2} + a_{3} + a_{4}$,
		\item[] $\beta_{2} = -a_{1} - a_{2} + a_{3} + a_{5}$,
		\item[] $\gamma_{1} = -a_{1} + a_{2} + a_{3} - a_{4}$,
		\item[] $\gamma_{2} = a_{1} + a_{2} + a_{3} - a_{4}$,
		\item[] $\delta_{1} = a_{2} - a_{3} - a_{4} + a_{5}$,
		\item[] $\delta_{2} = -a_{1} + a_{2} - a_{3} + a_{5}$,
		\item[] $\epsilon_{1} = -a_{1} + a_{3} + a_{4} - a_{5}$,
		\item[] $\epsilon_{2} = a_{1} + a_{3} + a_{4} - a_{5}$.
	\end{enumerate}	
\end{multicols}
It is easy to see that
\begin{equation}\label{Inequalities}
	\left.
	\begin{array}{ll}
		\ 0 < x_{i} < x_{i+1} \ \text{for} \ i = 1, 2, \ldots, 8;\ &   y_{1} < z_{1},\\
		\  y_{i} < y_{i+1} \ \text{for} \ i = 1, 2, \ldots, 6;\ &  \{-\gamma_{1}, \gamma_{1}\} < x_{1} < \beta_{1} < x_{3},\\
		z_{1} < z_{2} < z_{3} < z_{4} < z_{5}, & -\delta_{1} < \alpha_{1} <  \alpha_{2} < \alpha_{3},\\
		-\gamma_{1} < \beta_{1} < \beta_{2}, &   \gamma_{1} < \gamma_{2},\\
		\delta_{1} < \delta_{2},& \epsilon_{1} < \epsilon_{2},\\
		x_{4} < y_{7} < x_{8}, &   x_{4} < y_{6} < x_{7}, \\
		x_{4} < z_{5} < x_{6}, & y_{1} < y_{2} < x_{6}, \\
		\alpha_{2} < y_{3}, &  z_{5} < y_{6}, \\
		x_{2} < z_{3}, &  x_{3} < z_{4} < y_{5}, \\
		x_{3} < y_{5} <  y_{6} < x_{7}, &  \beta_{2} < z_{4}, \\
		\delta_{2} < z_{3}, & x_{1} < \delta_{2}. \\

	\end{array}
	\right\}
\end{equation}
In (\ref{Inequalities}),  $\{-\gamma_{1}, \gamma_{1}\} = \{0\}$ if $\gamma_{1}=0$, and $\{-\gamma_{1},\gamma_{1}\}$ is a symmetric set of cardinality two if $\gamma_{1} \neq 0$.

\begin{lemma}\label{main lemma 1}
	Let $A = \{a_{1}, a_{2}, a_{3}, a_{4}, a_{5}\}$ be a set of positive integers, where $a_{1} < a_{2} < a_{3} < a_{4} < a_{5}$. Let $a_{i} - a_{i - 1}  \neq  2a_{1}$ for exactly one  ${i} \in [2,5]$ and $a_{i} - a_{i-1} = 2a_{1}$ for the remaining $i \in [2,5]$. Then $\left|4_{\pm}^{\wedge}A \right| \geq 26$.
\end{lemma}

\begin{proof}
	
	Since signed sumset is symmetric, it suffices to prove that there are at least $13$ positive integers in $4_{\pm}^{\wedge}A$.
	
	\noindent\textbf{Case A}  $(a_{2} - a_{1} \neq 2a_{1}$ and $a_{5} - a_{4} = a_{4} - a_{3} = a_{3} - a_{2} = 2a_{1})$. In this case,
	\begin{itemize}
		\begin{multicols}{2}
			\item [] $\gamma_{1} = -a_{1} + a_{2} + a_{3} - a_{4} = -3a_{1} + a_{2} $,		
			\item [] $\gamma_{2} = a_{1} + a_{2} + a_{3} - a_{4} = a_{2} - a_{1}$,
			\item [] $x_{1} = -a_{1} + a_{2} - a_{3} + a_{4} = a_{2} + a_{1}$,
			\item [] $x_{2} = a_{1} + a_{2} - a_{3} + a_{4} = 3a_{1} + a_{2}$,
			\item[] $\beta_{1} = -a_{1} - a_{2} + a_{3} + a_{4} = 3a_{1} + a_{3}$, 
			\item [] $x_{3} = a_{1} - a_{2} + a_{3} + a_{4} = 3a_{1} + a_{4}$,
			\item [] $\alpha_{1} = -a_{2} + a_{3} - a_{4} + a_{5} = 4a_{1}$.	
		\end{multicols}
	\end{itemize}
	Clearly,
	\begin{equation}\label{eqn-A_{1}}
		0 \neq \gamma_{1} <\gamma_{2} < x_{1}, \ -\gamma_{1} < \alpha_{1} < x_{2} <\beta_{1} < x_{3}  \ \text{and} \ \alpha_{1} \neq x_{1}.
	\end{equation}
	If $a_{2} - a_{1} = a_{1}$, then $A = a_{1} \ast \{1,2,4,6,8\}$, and so $\left|4_{\pm}^{\wedge}A \right| \geq 26$. If $a_{2} - a_{1} \neq a_{1}$, then $ \gamma_{2} \neq - \gamma_{1}$. Note that, if $\gamma_{1} > 0$, then we have the following $12$ distinct elements of $4_{\pm}^{\wedge}A$
	\begin{equation}\label{eqn-A_{2}}
		0  < \gamma_{1} <  \gamma_{2}  < x_{1} < x_{2} < \beta_{1} < x_{3} < x_{4} < x_{5} < x_{6} < x_{7} < x_{8} < x_{9} ,
	\end{equation}
	and if $-\gamma_{1} > 0$, then we have the following $12$ distinct elements of $4_{\pm}^{\wedge}A$
	\begin{equation}\label{eqn-A_{3}}
		0  < \{-\gamma_{1},  \gamma_{2}\}  < x_{1} < x_{2} < \beta_{1} < x_{3} < x_{4} <x_{5} <x_{6} < x_{7}  < x_{8} < x_{9}.
	\end{equation}
	Therefore, in each case, we have at least $12$ positive integers in  $4_{\pm}^{\wedge}A$. Next we show the existence of at least one more positive integer in $4_{\pm}^{\wedge}A$ different from the elements listed in  (\ref{eqn-A_{2}}) and (\ref{eqn-A_{3}}). Consider the positive integer $\alpha_{1}$. Since $-\gamma_{1} < \alpha_{1} < x_{2}$ and  $\alpha_{1} \neq x_{1}$,  we have the following cases:\\
	\noindent$\clubsuit$ If $\gamma_{1} > 0$, $\alpha_{1} \neq \gamma_{1}$ and $\alpha_{1} \neq \gamma_{2}$, then  we can include $\alpha_{1}$ in the list (\ref{eqn-A_{2}}) to get the required number of elements listed below
	\[0  < \{\gamma_{1},  \gamma_{2},  \alpha_{1}, x_{1} \} < x_{2} < \beta_{1} < x_{3} < x_{4} < x_{5} <x_{6} < x_{7} < x_{8} < x_{9}.\]
	$\clubsuit$ If $ -\gamma_{1} > 0$ and $\alpha_{1} \neq \gamma_{2}$, then we can add $\alpha_{1}$ in  (\ref{eqn-A_{3}}) to get required number of elements
	\[0 < \{-\gamma_{1},  \gamma_{2},  \alpha_{1}, x_{1} \} < x_{2} < \beta_{1} < x_{3} < x_{4} < x_{5} <x_{6} < x_{7} < x_{8} < x_{9}.\]	
	$\clubsuit$ If  $\alpha_{1} = \gamma_{1}$, then  $a_{2} = 7a_{1}$. This gives  $A = a_{1}\ast\{1,7,9,11,13\}$, and so $|4_{\pm}^{\wedge}A| \geq 26$.\\
	$\clubsuit$ If   $\alpha_{1} = \gamma_{2}$, then  $a_{2} = 5a_{1}$. This gives  $A = a_{1}\ast\{1,5,7,9,11\}$, and so $|4_{\pm}^{\wedge}A| \geq 26$.
	
	In each case either $|4_{\pm}^{\wedge}A| \geq 26$ or there are at least $13$ distinct positive elements in  $4_{\pm}^{\wedge}A$. \\
	
	\noindent \textbf{Case B}  $(a_{3} - a_{2}  \neq  2a_{1}$ and $a_{5} - a_{4} = a_{4} - a_{3} = a_{2} - a_{1} = 2a_{1})$. In this case,
	\begin{multicols}{2}
		\begin{itemize}
			\item [] $a_{2} = 3a_{1}$, $a_{3} \neq 5a_{1}$,
			
			\item [] $\gamma_{2} = a_{1} + a_{2} + a_{3} - a_{4} = 2a_{1}$,
			\item [] $x_{1} = -a_{1} + a_{2} - a_{3} + a_{4} = 4a_{1}$,
			\item [] $\alpha_{1} = -a_{2} + a_{3} - a_{4} + a_{5} = -a_{1} + a_{3} $,
			\item [] $\alpha_{2} = -a_{2} - a_{3} + a_{4} + a_{5} = -a_{1} + a_{5} $,			
			\item [] $y_{3} = -a_{1} - a_{3} + a_{4} + a_{5}$,
			\item [] $x_{7} = a_{1} + a_{2} + a_{4} + a_{5}$,
			\item [] $y_{7} = -a_{1} + a_{3} + a_{4} + a_{5}$,
			\item [] $x_{8} = a_{1} + a_{3} + a_{4} + a_{5}$,		
			\item [] $y_{6} = -a_{1} + a_{2} + a_{4} + a_{5}$,	
			\item [] $z_{2} = a_{1} + a_{2} -a_{4} + a_{5} = 6a_{1}$.
		\end{itemize}
	\end{multicols}
	\noindent Clearly,
	\begin{equation}\label{Eqn-B_{1}}
		0 < \gamma_{2} < \{x_{1}, \  \alpha_{1}\}, x_{1} \neq \alpha_{1}, \ y_{7} \neq x_{7}, \  x_{1} < z_{2},
	\end{equation}
	and
	\begin{equation}\label{Eqn-B_{2}}
		x_{1} < z_{2} = 6a_{1} = 3a_{1} + a_{2} < 3a_{1} + a_{3} = a_{1} + a_{4} = -a_{1} + a_{5} = \alpha_{2}.
	\end{equation}
	It follows from (\ref{Eqn-B_{1}}), (\ref{Eqn-B_{2}}) and (\ref{Inequalities}) that
	\begin{equation}\label{Eqn-B_{3}}
		0 < \gamma_{2} < \{x_{1}, \alpha_{1}\} < \alpha_{2} < y_{3} < y_{4} < y_{5} < y_{6} <\{ y_{7}, x_{7}\} < x_{8} < x_{9}.
	\end{equation}
	Therefore, we have at least $12$ positive integers in $4_{\pm}^{\wedge}A$. Now we show the existence of at least one more positive integer in  $4_{\pm}^{\wedge}A$ different from the listed integers in $(\ref{Eqn-B_{3}})$. Since $x_{1} < z_{2} <  \alpha_{2}$, we have the following cases:\\
	\noindent $\clubsuit$ If $z_{2} \neq \alpha_{1}$, then  we have at least $13$ positive integers in $4_{\pm}^{\wedge}A$, listed below
	\[	0  < \gamma_{2} < \{x_{1}, \alpha_{1}, z_{2}\} < \alpha_{2} < y_{3} < y_{4} < y_{5} < y_{6} <\{ y_{7}, x_{7}\} < x_{8} < x_{9}.\]
	\noindent $\clubsuit$ If $z_{2} = \alpha_{1}$, then we have $a_{3} = 7a_{1}$ and which implies $A = a_{1} \ast \{1,3,7,9,11\}$, and so $|4_{\pm}^{\wedge}A| \geq 26$.\\
	
	\noindent\textbf{Case C} $(a_{4} - a_{3} \neq 2a_{1}$ and $a_{5} - a_{4} = a_{3} - a_{2} = a_{2} - a_{1} = 2a_{1})$. In this case,
	\begin{multicols}{2}
		\begin{itemize}
			\item [] $a_{2} = 3a_{1}, a_{3} = 5a_{1} < a_{4}, a_{4} \neq 7a_{1}$,
			
			\item [] $\gamma_{1} = -a_{1} + a_{2}  + a_{3} - a_{4} = 7a_{1} - a_{4}$,
			
			\item [] $x_{1} = -a_{1} + a_{2} - a_{3} + a_{4} = -3a_{1} + a_{4}$,	
			\item [] $x_{2} = a_{1} + a_{2} - a_{3} + a_{4} = -a_{1} + a_{4}$,
			\item [] $\beta_{1} = -a_{1} - a_{2} + a_{3} + a_{4} = a_{1} + a_{4}$,	
			\item [] $x_{3} = a_{1} - a_{2} + a_{3} + a_{4} = 3a_{1} + a_{4}$,
			\item [] $z_{1} = a_{1} + a_{2} + a_{4} - a_{5} = 2a_{1}$,		
			\item [] $x_{6} = a_{1} + a_{2} + a_{3} +  a_{5} =  9a_{1}  + a_{5}$,
			\item [] $y_{6} = -a_{1} + a_{2} + a_{4} + a_{5} = 2a_{1} + a_{4} + a_{5}$.
			
		\end{itemize}
	\end{multicols}
	\noindent Clearly,
	\begin{equation}\label{Eqn-C_{1}}
		-\gamma_{1} < x_{1} < x_{2} < \beta_{1} < x_{3} \ \text{and} \  x_{6} \neq y_{6}.
	\end{equation}
	Also
	\begin{equation}\label{Eqn-C_{2}}
		\gamma_{1} = -a_{1} + a_{2} + a_{3} - a_{4} = 2a_{1} + a_{3} - a_{4} < 2a_{1} = z_{1} <-3a_{1} + a_{4} = x_{1},
	\end{equation}
	and
	\begin{equation}\label{Eqn-C_{3}}
		x_{5} = a_{1} + a_{2} + a_{3} + a_{4} < a_{1} + a_{2} + 2a_{4} = 2a_{1} + a_{4} + a_{5} = y_{6}.
	\end{equation}
	Since $\gamma_{1} \neq 0$, therefore, it follows   from (\ref{Eqn-C_{1}}), (\ref{Eqn-C_{2}}), (\ref{Eqn-C_{3}}) and (\ref{Inequalities})  that
	\[0  < \{-\gamma_{1} \ \text{or} \ \gamma_{1}\} < x_{1} < x_{2} < \beta_{1} < x_{3} <x_{4} < x_{5} < \{x_{6}, y_{6}\} < x_{7} < x_{8} < x_{9}\]
	and	
	\[ \gamma_{1} < z_{1} < x_{1}.\]
	Therefore, we have at least $13$ positive integers in $4_{\pm}^{\wedge}A$, except in the case when $z_{1} = -\gamma_{1}$. If $z_{1} = -\gamma_{1}$, then we have  $a_{4} = 9a_{1}$, which implies that $A = a_{1} \ast \{1,3,5,9,11\}$, and so $|4_{\pm}^{\wedge}A| \geq 26$.\\
	
	\noindent\textbf{Case D}  $(a_{5} - a_{4} \neq 2a_{1}$ and $a_{4} - a_{3} = a_{3} - a_{2} = a_{2} - a_{1} = 2a_{1})$. In this case,
	\begin{multicols}{2}
		\begin{itemize}
			\item [] $a_{2} = 3a_{1}, a_{3} = 5a_{1}, a_{4} = 7a_{1}$,
			\item [] $a_{5} \neq 9a_{1}, a_{5} > 7a_{1}$,
			
			\item[] $\gamma_{2} = a_{1} + a_{2} + a_{3} - a_{4}  = 2a_{1}$,
			\item[] $x_{1} = -a_{1} + a_{2}  - a_{3} + a_{4} = 4a_{1}$,
			\item[] $x_{2} = a_{1} + a_{2} - a_{3} + a_{4} = 6a_{1}$,
			\item[] $\beta_{1} = -a_{1} - a_{2} + a_{3} + a_{4} = 8a_{1}$,	
			\item[] $x_{3} = a_{1} - a_{2} + a_{3} + a_{4} = 10a_{1}$,
			\item[] $x_{4} = -a_{1} + a_{2} + a_{3} + a_{4} = 14a_{1}$,
			\item[] $z_{5} = -a_{1} + a_{2} + a_{3} + a_{5} = 7a_{1} + a_{5}$,
			\item[] $x_{5} = a_{1} + a_{2} + a_{3} + a_{4} = 16a_{1}$,
			\item[] $x_{6} = a_{1} + a_{2} + a_{3} + a_{5} = 9a_{1} + a_{5}$,
			\item [] $y_{2} = -a_{1} + a_{2} - a_{4} + a_{5} = 2a_{1} - a_{4} + a_{5}$.
		\end{itemize}
	\end{multicols}
	\noindent	Clearly, we have 	
	\begin{equation}\label{Eqn-D_{1}}
		0 < \gamma_{2} < x_{1} < x_{2} < \beta_{1} < x_{3} < x_{4} < \{z_{5},x_{5}\} < x_{6} < x_{7} < x_{8} < x_{9}.
	\end{equation}
	Therefore, we have atleast $12$ positive integers in $4_{\pm}^{\wedge}A$. Next, we show the existence of at least one more positive integer in $4_{\pm}^{\wedge}A$ different from the elements listed in $(\ref{Eqn-D_{1}})$. Since  $y_{2} \neq x_{1}$ and $ \gamma_{2} < y_{2} < z_{5} < x_{6}$, we have the following cases:\\
	$\clubsuit$ If $y_{2}\notin \{x_{2}, \beta_{1}, x_{3}, x_{4}, x_{5}\}$, then we have at least $13$ positive  integers in $4_{\pm}^{\wedge}A $, listed below
	\[0  < \gamma_{2} < \{x_{1}, x_{2}, \beta_{1}, x_{3}, x_{4}, z_{5}, x_{5}, y_{2}\} < x_{6} < x_{7} < x_{8} < x_{9}.\]
	$\clubsuit$ If $y_{2} = x_{2}$, then  $a_{5} = 11a_{1}$. This gives $A = a_{1}  \ast \{1, 3, 5, 7, 11\} $, and so $|4_{\pm}^{\wedge}A| \geq 26$.\\
	$\clubsuit$ If $y_{2} = \beta_{1}$, then $a_{5} = 13a_{1}$. This gives $A = a_{1}  \ast \{1, 3, 5, 7, 13\} $, and so $|4_{\pm}^{\wedge}A| \geq 26$.\\
	$\clubsuit$ If $y_{2} = x_{3}$, then $a_{5} = 15a_{1}$. This gives $A = a_{1}  \ast \{1, 3, 5, 7, 15\} $, and so $|4_{\pm}^{\wedge}A| \geq 26$.\\
	$\clubsuit$ If $y_{2} = x_{4}$, then $a_{5} = 19a_{1}$. This gives $A = a_{1}  \ast \{1, 3, 5, 7, 19\} $, and so $|4_{\pm}^{\wedge}A| \geq 26$.\\
	$\clubsuit$ If $y_{2} = x_{5}$, then $a_{5} = 21a_{1}$. This gives $A = a_{1}  \ast \{1, 3, 5, 7, 21\} $, and so $|4_{\pm}^{\wedge}A| \geq 26$.
\end{proof}

\begin{lemma}\label{main lemma 2}		
	Let $ A = \{a_{1}, a_{2}, a_{3}, a_{4}, a_{5}\}$ be a set of positive integers, where $a_{1} < a_{2} < a_{3} < a_{4} < a_{5}$. Let $a_{i} - a_{i - 1} \neq 2a_{1}$  for exactly two  $ i \in [2,5]$  and $a_{i} - a_{i-1} = 2a_{1}$ for the remaining two $i \in [2,5]$. Then $\left|4_{\pm}^{\wedge}A \right| \geq 26$.
\end{lemma}

\begin{proof} To prove $\left|4_{\pm}^{\wedge}A \right| \geq 26$, it is sufficient to prove that there are at least $13$ positive integers in $4_{\pm}^{\wedge}A$.
	
	\noindent \textbf{Case A} $(a_{5} - a_{4}  = a_{4} - a_{3} = 2a_{1}$,   $a_{3} - a_{2}  \neq 2a_{1}$ and  $a_{2} - a_{1}  \neq  2a_{1})$. In this case,
	\begin{multicols}{2}
		\begin{itemize}
			\item [] $z_{1} = a_{1} + a_{2} + a_{4} - a_{5} = -a_{1} + a_{2}$,
			\item [] $\epsilon_{2} = a_{1} + a_{3} + a_{4} - a_{5} =-a_{1} + a_{3}$,
			\item [] $x_{1} = -a_{1} + a_{2} - a_{3} + a_{4} = a_{1} + a_{2}$,
			\item [] $\beta_{1} = -a_{1} - a_{2} + a_{3} + a_{4} = a_{1} - a_{2} + 2a_{3} $,
			\item []$\beta_{2} = -a_{1} - a_{2} + a_{3} + a_{5} = a_{1} -a_{2} + a_{3} +a_{4}$,
			\item []$y_{4} = -a_{1} - a_{2} + a_{4} + a_{5} = a_{1} - a_{2} + 2a_{4}$,
			\item []$x_{5} = a_{1} + a_{2} + a_{3} + a_{4} = 3a_{1} + a_{2} + 2a_{3}$,
			\item []$x_{6} = a_{1} + a_{2} + a_{3} + a_{5}$,
			\item []$x_{7} = a_{1} + a_{2} + a_{4} + a_{5}$,
			\item []$y_{7} = -a_{1} + a_{3} + a_{4} + a_{5}$,
			\item []$x_{8} = a_{1} + a_{3} + a_{4} + a_{5}$,
			\item []$x_{9} = a_{2} + a_{3} + a_{4} + a_{5}$,
			\item []$x_{4} = -a_{1} + a_{2} + a_{3} + a_{4}$,
			\item []$z_{3} = a_{1} + a_{2} - a_{3} + a_{5} =a_{1}+a_{2}+4a_{1}= 5a_{1} + a_{2}$.
			
		\end{itemize}
	\end{multicols}
	\noindent	Clearly, $\epsilon_{2} \neq x_{1}$, $x_{7} \neq y_{7}$, $x_{6} < y_{7}$ and $x_{5}  >  y_{4}$. Therefore, we have
	\begin{equation}\label{Lemma 3.4 eq-1}
		0 <z_{1} < \{\epsilon_{2}, x_{1}\} < \beta_{1} < \beta_{2} < y_{4} < x_{5} < x_{6} <\{x_{7}, y_{7}\} <x_{8} < x_{9},
	\end{equation}
	\[ \beta_{2} < x_{4} < x_{5},\]
	and
	\begin{equation*}
		x_{1} < z_{3} < y_{4}.
	\end{equation*}
	In (\ref{Lemma 3.4 eq-1}), we have $12$ distinct positive integers of $4_{\pm}^{\wedge}A$. Now, we show that there exists at least one more positive integer in $4_{\pm}^{\wedge}A$ distinct from the listed integers in (\ref{Lemma 3.4 eq-1}). Consider the positive integers $x_{4}$ and $z_{3}$. If $ x_{4} \neq y_{4}$, then  we have at least $13$ positive integers in $4_{\pm}^{\wedge}A$ listed below
	\[	0 <z_{1} < \{\epsilon_{2}, x_{1}\} < \beta_{1} < \beta_{2} < \{ y_{4}, x_{4} \} < x_{5} < x_{6} <\{x_{7}, y_{7}\} <x_{8} < x_{9}.\]
	Assume $ x_{4} = y_{4}$.  Then $a_{2}=2a_{1}$. Now, consider following cases:
	\begin{enumerate}
		\item  If  $z_{3} \notin \{\epsilon_{2}, \beta_{1}, \beta_{2}\}$, then we have at least $13$ positive integers in $4_{\pm}^{\wedge}A$, listed below
		\[	0 <z_{1} < \{\epsilon_{2}, x_{1},  \beta_{1}, \beta_{2}, z_{3} \} <  y_{4} < x_{5} < x_{6} <\{x_{7}, y_{7}\} <x_{8} < x_{9}.\]
		\item  If $z_{3} = \epsilon_{2}$, then  $a_{3} = 8a_{1}$. This gives $A =a_{1} \ast \{1,2,8,10,12\}$, and so $\left|4_{\pm}^{\wedge}A \right| \geq 26$.
		
		\item  If $z_{3} = \beta_{1}$, then  $a_{3} = 4a_{1}$. This gives $ A =a_{1} \ast \{1,2,4,6,8\}$, and so $\left|4_{\pm}^{\wedge}A \right| \geq 26$.
		
		\item  If $z_{3} = \beta_{2}$, then   $a_{3} = 3a_{1}$. This gives $ A =  a_{1} \ast \{1,2,3,5,7\}$, and so $\left|4_{\pm}^{\wedge}A \right| \geq 26$.
	\end{enumerate}
	
	\noindent \textbf{Case B}  $(a_{5} - a_{4}  =  a_{3} - a_{2}  = 2a_{1}$, $a_{4} - a_{3} \neq 2a_{1}$,    and  $a_{2} - a_{1}  \neq  2a_{1})$. In this case,
	
	\begin{multicols}{2}
		\begin{itemize}
			\item [] $a_{5} - a_{3} > a_{5} - a_{4} = 2a_{1}$,
			\item [] $a_{4} - a_{2} > a_{3} -a_{2} = 2a_{1}$,
			
			\item [] $z_{1} = a_{1} + a_{2} + a_{4} - a_{5} = a_{2} - a_{1}$,
			\item [] $x_{1} = -a_{1} + a_{2} - a_{3} + a_{4} = -3a_{1} + a_{4}$,
			\item [] $x_{2} = a_{1} + a_{2} - a_{3} + a_{4} = -a_{1} + a_{4}$,
			\item [] $\beta_{1} = - a_{1} - a_{2} + a_{3} + a_{4} = a_{1} + a_{4}$,
			\item [] $x_{3} = a_{1} - a_{2} + a_{3} + a_{4} = 3a_{1} + a_{4}$,
			\item [] $x_{4} = -a_{1} + a_{2} + a_{3} + a_{4} = a_{1} + 2a_{2} + a_{4}$,
			\item [] $x_{5} = a_{1} + a_{2} + a_{3} + a_{4} = 3a_{1} + 2a_{2} + a_{4}$,
			\item [] $x_{6} = a_{1} + a_{2} + a_{3} + a_{5} = 3a_{1} + 2a_{2} + a_{5}$,
			\item [] $y_{6} = -a_{1} + a_{2} + a_{4} + a_{5}  $,
			\item [] $x_{7} = a_{1} + a_{2} + a_{4} + a_{5}$,
			\item [] $x_{8} = a_{1} + a_{3} + a_{4} + a_{5}$,
			\item [] $x_{9} = a_{2} + a_{3} + a_{4} + a_{5}$,
			\item [] $z_{2} = a_{1} + a_{2} - a_{4} + a_{5} = 3a_{1} + a_{2}$,
			\item [] $z_{4}  = a_{1} - a_{2} + a_{3} + a_{5} = 3a_{1} + a_{5}$.
		\end{itemize}
	\end{multicols}
	\noindent		Note also that	$x_{6} = a_{1} + a_{2} + a_{3} + a_{5}  \neq -a_{1} + a_{2} + a_{4} + a_{5} = y_{6} $. Therefore, we have
	\begin{equation}\label{Lemma 3.4 eq-2}
		0  < z_{1} < x_{1} <  x_{2} < \beta_{1} < x_{3} < x_{4} <x_{5} <\{x_{6}, y_{6}\} < x_{7} <x_{8} <x_{9}
	\end{equation}
	and
	\begin{equation*}
		z_{1} < z_{2} < x_{3} < z_{4} < x_{5}.
	\end{equation*}
	In (\ref{Lemma 3.4 eq-2}), we have $12$ distinct positive integers of $4_{\pm}^{\wedge}A$. Next, we show  that there is at least one more positive integer in   $4_{\pm}^{\wedge}A$ distinct from the elements listed in (\ref{Lemma 3.4 eq-2}). Consider  the following cases:\\
	$\clubsuit$ If $z_{4}  \neq  x_{4}$, then we have at least $13$ positive  integers in $4_{\pm}^{\wedge}A$, listed below
	\[	0 < z_{1} < x_{1} <  x_{2} < \beta_{1} < x_{3} <\{ x_{4}, z_{4} \} < x_{5} < \{x_{6}, y_{6}\} < x_{7} <x_{8} <x_{9}.\]
	$\clubsuit$ If $z_{4} = x_{4}$ and $z_{2} \neq x_{1}$, then   $a_{2} = 2a_{1}$, $a_{3} = 4a_{1}$, $z_{2} < \beta_{1}$ and $z_{2} \neq x_{2}$. Therefore,  we have at least $13$ positive integers in $4_{\pm}^{\wedge}A$, listed below
	\[	0 < z_{1} < \{x_{1}, x_{2}, z_{2} \}< \beta_{1} < x_{3} < x_{4} < x_{5} < \{x_{6}, y_{6}\} < x_{7} <x_{8} <x_{9}.\]
	$\clubsuit$ If $z_{4} = x_{4}$ and $z_{2} = x_{1}$, then $a_{2} = 2a_{1}$, $a_{3} = 4a_{1}$, $a_{4} = 8a_{1}$. This gives $A=a_{1}\ast\{1,2,4,8,10\}$, and so $\left|4_{\pm}^{\wedge}A \right| \geq 26$.\\
	
	\noindent \textbf{Case C} $(a_{4} - a_{3} = a_{3} - a_{2}  =2a_{1}$, $a_{5} - a_{4}  \neq  2a_{1} $ and  $a_{2} - a_{1}  \neq  2a_{1})$. In this case,
	\begin{multicols}{2}
		\begin{itemize}
			\item [] $a_{2} \neq 3a_{1}$,
			\item [] $\gamma_{1} = -a_{1} + a_{2} + a_{3} - a_{4}  \neq 0$,
			\item [] $\gamma_{2} = a_{1} + a_{2} + a_{3} - a_{4} = -a_{1} + a_{2}$,
			\item [] $x_{1} = -a_{1} + a_{2} - a_{3} + a_{4} = -3a_{1} + a_{4}$,
			\item [] $x_{2} = a_{1} + a_{2} - a_{3} + a_{4} = -a_{1} + a_{4}$,
			\item [] $\beta_{1} = -a_{1} - a_{2} + a_{3} + a_{4} =a_{1} + a_{4} $,
			\item [] $x_{3} = a_{1} - a_{2} + a_{3} + a_{4} = 3a_{1} + a_{4}$,
			\item [] $x_{4} = -a_{1} + a_{2} + a_{3} + a_{4} = a_{1} + 2a_{2} + a_{4}$,
			\item [] $x_{5} = a_{1} + a_{2} + a_{3} + a_{4} = 3a_{1} +2a_{2} +a_{4}$,
			\item [] $x_{6} = a_{1} + a_{2} + a_{3} + a_{5} = 3a_{1} + 2a_{2} + a_{5}$,
			\item [] $z_{5} = -a_{1} + a_{2} + a_{3} + a_{5} = a_{1} + 2a_{2} + a_{5}$,
			\item [] $y_{4} = -a_{1} - a_{2} + a_{4} + a_{5}$.
		\end{itemize}
	\end{multicols}
	\noindent Clearly,
	\begin{equation}\label{Lemma 3.4 eq-3}
		0< \gamma_{2} < x_{1} < x_{2} < \beta_{1} < x_{3} < x_{4} < \{x_{5}, z_{5}\} < x_{6} < x_{7} < x_{8} < x_{9}.
	\end{equation}
	In $(\ref{Lemma 3.4 eq-3})$, we have  $12$ distinct positive integers of $4_{\pm}^{\wedge}A$. Next, we show the existence of at least one more positive integer in $4_{\pm}^{\wedge}A$ different from the elements listed in $(\ref{Lemma 3.4 eq-3})$.
	Now, consider the following cases:\\
	$\clubsuit$ If $\gamma_{1} > 0$, then  we have  at least $13$ positive integers in  $4_{\pm}^{\wedge}A$, listed below
	\[0< \gamma_{1} < \gamma_{2} < x_{1} < x_{2} < \beta_{1} < x_{3} < x_{4} < \{x_{5}, z_{5}\} < x_{6} < x_{7} < x_{8} < x_{9}.\]
	$\clubsuit$ If $-\gamma_{1} > 0$ and  $\gamma_{2}  \neq  -\gamma_{1}$, then we have  at least $13$ positive integers in  $4_{\pm}^{\wedge}A$, listed below
	\[0< \{ -\gamma_{1}, \gamma_{2} \} < x_{1} < x_{2} < \beta_{1} < x_{3} < x_{4} < \{x_{5}, z_{5}\} < x_{6} < x_{7} < x_{8} < x_{9}.\]
	$\clubsuit$ If $-\gamma_{1} > 0$, $\gamma_{2}  = -\gamma_{1} $ and $y_{4} \neq x_{5}$, then $a_{2} = 2a_{1}$,  $a_{3} = 4a_{1}, a_{4} = 6a_{1} < a_{5}$ and $y_{4} =  -a_{1} - a_{2} + a_{4} + a_{5} = 3a_{1} + a_{5}$. It is easy to see that $x_{3} < y_{4} < z_{5}  < x_{6}$ and $y_{4} \neq x_{4}$. Therefore, we have at least $13$ positive integers in  $4_{\pm}^{\wedge}A$, listed below
	\[	0< \gamma_{2} < x_{1} < x_{2} < \beta_{1} < x_{3} <\{ x_{4}, x_{5}, z_{5}, y_{4}\} < x_{6} < x_{7} < x_{8} < x_{9}.\]
	$\clubsuit$ If $-\gamma_{1} > 0$, $\gamma_{2}  = -\gamma_{1} $ and $y_{4} = x_{5}$, then $a_{2} = 2a_{1}$,  $a_{3} = 4a_{1}, a_{4} = 6a_{1}$, $a_{5} = 10a_{1}$, and so  $A = a_{1} \ast \{1,2,4,6,10\}$.\\
	In each of these cases, we have $\left|4_{\pm}^{\wedge}A \right| \geq 26$.\\
	
	\noindent \textbf{Case D} $(a_{5} - a_{4} =a_{2} - a_{1}   = 2a_{1} $, $a_{4} - a_{3}\neq 2a_{1}$ and    $a_{3} - a_{2}  \neq 2a_{1})$. In this case,
	\begin{multicols}{2}
		\begin{itemize}
			\item [] $a_{2} = 3a_{1}, a_{3} \neq 5a_{1}, a_{3} > 3a_{1}$,
			
			\item [] $z_{1} = a_{1} + a_{2} + a_{4} - a_{5}  = 2a_{1} $,
			\item [] $y_{2} = -a_{1} + a_{2} - a_{4} + a_{5} = 4a_{1}$,
			\item [] $x_{2} = a_{1} + a_{2} - a_{3} + a_{4} = 4a_{1} -a_{3} +a_{4}$,
			\item [] $z_{2} = a_{1} + a_{2} - a_{4} + a_{5} = 6a_{1}$,
			\item [] $z_{3} = a_{1} + a_{2} - a_{3} + a_{5} = 6a_{1} - a_{3} +a_{4}$,
			\item [] $z_{4} = a_{1} - a_{2} + a_{3} + a_{5} = -2a_{1} + a_{3} +a_{5} $,
			\item [] $y_{5} = a_{1} - a_{2} + a_{4} + a_{5} =-2a_{1} + a_{4} + a_{5}$,
			\item [] $y_{6} =  -a_{1} + a_{2} + a_{4} + a_{5} =  2a_{1} + a_{4} + a_{5}$,
			\item [] $x_{7} = a_{1} + a_{2} + a_{4} + a_{5} = 4a_{1} + a_{4} +a_{5}$,
			\item [] $y_{7} = -a_{1} + a_{3} + a_{4} + a_{5}$,
			\item [] $x_{1} = -a_{1} + a_{2} - a_{3} + a_{4} = 2a_{1} -a_{3} +a_{4} $,
			\item [] $\beta_{2} = -a_{1} - a_{2} + a_{3} + a_{5} = -2a_{1} +a_{3} + a_{4}$.
		\end{itemize}
	\end{multicols}
	\noindent	Clearly,
	\begin{equation}\label{Lemma 3.4 eq-4}
		0 < z_{1} < y_{2} < \{x_{2},z_{2}\} < z_{3} < z_{4} < y_{5} < y_{6} < \{x_{7},y_{7}\} < x_{8} < x_{9},
	\end{equation}
	\begin{equation}\label{Lemma 3.4 eq-5}
		x_{1} \neq y_{2}, z_{1} < x_{1} <x_{2} < z_{3} \ \text{and} \ x_{2} <  \beta_{2} < z_{4}.
	\end{equation}
	We have $12$ distinct positive integers in the list (\ref{Lemma 3.4 eq-4}), now we show the existence of at least one more positive integer different from the listed integers in (\ref{Lemma 3.4 eq-4}). Consider $x_{1}$ and $\beta_{2}$, and consider the following cases:\\
	$\clubsuit$	If $x_{1} \neq z_{2}$, then  we have at least $13$ positive integers in $4_{\pm}^{\wedge}A$, listed below
	\[0 < z_{1} < \{y_{2}, x_{1}, x_{2}, z_{2}\} < z_{3} < z_{4} < y_{5} < y_{6} < \{x_{7},y_{7}\} < x_{8} < x_{9}.\]
	$\clubsuit$ If $x_{1} = z_{2}$ and $\beta_{2} \neq z_{3}$, then  we have  at least $13$ positive integers in $4_{\pm}^{\wedge}A$, listed below
	\[	0 < z_{1} < y_{2} < x_{1}< x_{2} < \{ z_{3}, \beta_{2}\} < z_{4} < y_{5} < y_{6} < \{x_{7},y_{7}\} < x_{8} < x_{9}.\]
	$\clubsuit$ If $x_{1} = z_{2}$ and $\beta_{2} = z_{3}$, then $a_{3} = 4a_{1}, a_{4} = 8a_{1}$ and $a_{5} = 10a_{1}$, which implies that $A=a_{1}\ast\{1,3,4,8,10\}$, and so $\left|4_{\pm}^{\wedge}A \right| \geq 26$.\\
	
	\noindent \textbf{Case E} $(a_{4} - a_{3} =  a_{2} - a_{1} = 2a_{1}, a_{3} -a_{2} \neq 2a_{1}$ and $a_{5} - a_{4} \neq 2a_{1})$. In this case,
	\begin{multicols}{2}		
		\begin{itemize}
			\item [] $a_{2} = 3a_{1}$,
			\item [] $a_{4} =2a_{1} + a_{3} > 2a_{1} + a_{2} = 5a_{1}$,
			\item [] $\gamma_{2} = a_{1} + a_{2} + a_{3} - a_{4} = 2a_{1}$,
			\item [] $x_{1} = -a_{1} + a_{2} - a_{3} + a_{4} = 4a_{1}$,
			\item [] $x_{2} = a_{1} + a_{2} -a_{3} + a_{4} = 6a_{1}$,
			\item [] $x_{3} = a_{1} -a_{2} + a_{3} + a_{4} = -2a_{1} + a_{3} + a_{4}$,
			\item [] $x_{4} = -a_{1} + a_{2} + a_{3} + a_{4} = 2a_{1} + a_{3} + a_{4}$,
			\item [] $x_{5} = a_{1} + a_{2} + a_{3} + a_{4} = 4a_{1} + a_{3} + a_{4}$,
			\item [] $z_{5} = -a_{1} + a_{2} + a_{3} + a_{5} = 2a_{1} + a_{3} + a_{5}$,
			\item [] $x_{6} = a_{1} + a_{2} + a_{3} + a_{5} = 4a_{1} + a_{3} +a_{5}$,
			\item [] $x_{7} = a_{1} + a_{2} + a_{4} + a_{5} = 4a_{1} + a_{4} +a_{5}$,
			\item [] $\beta_{1} = -a_{1} - a_{2} + a_{3} + a_{4} = -2a_{1} + 2a_{3}$,
			\item []$z_{4} = a_{1} - a_{2} + a_{3} + a_{5} = -2a_{1} + a_{3} +a_{5}$,
			\item []$ y_{7} = -a_{1} + a_{3} + a_{4} + a_{5} = a_{1} + 2a_{3} + a_{5}.$
			
		\end{itemize}	
	\end{multicols}
	\noindent	Clearly,
	\begin{equation*}\label{L2-eqn-A_{1}}
		0  < \gamma_{2} < x_{1}, ~ x_{5} \neq z_{5}, ~ x_{7} \neq y_{7},
	\end{equation*}
	and
	\[ y_{7} = a_{1} + 2a_{3} + a_{5} > a_{1} + a_{2} +  a_{3}  + a_{5} = x_{6}.\]
	Therefore, we have
	\begin{equation}\label{Lemma 3.4 eq-6}
		0  < \gamma_{2} < x_{1} < x_{2} < x_{3} < x_{4} < \{x_{5}, z_{5}\} < x_{6} < \{x_{7}, y_{7}\} < x_{8} < x_{9}
	\end{equation}		
	and \[ x_{1} < \beta_{1} < x_{3} < z_{4} < z_{5}.\]
	We have $12$ distinct positive integers in (\ref{Lemma 3.4 eq-6}). Next, we show the existence of at least one more positive integer different from the elements listed  in (\ref{Lemma 3.4 eq-6}). Consider  the following cases:\\
	$\clubsuit$ 	If $\beta_{1} \neq x_{2}$, then we have at least 13 positive integers in $4_{\pm}^{\wedge}A$, listed below
	\[0  < \gamma_{2} < x_{1} < \{ x_{2}, \beta_{1} \} < x_{3} < x_{4} < \{x_{5}, z_{5}\} < x_{6} < \{x_{7}, y_{7}\} < x_{8} < x_{9}.\]
	$\clubsuit$ If  $z_{4} \neq x_{4}$ and $z_{4} \neq x_{5} $, then we have at least 13 positive integers in $4_{\pm}^{\wedge}A$, listed below
	\[0  < \gamma_{2} < x_{1} < x_{2} < x_{3} < \{x_{4}, x_{5}, z_{4}, z_{5}\} < x_{6} < \{x_{7}, y_{7}\} < x_{8} < x_{9}.\]
	$\clubsuit$  If $\beta_{1} = x_{2}$ and $z_{4} = x_{4}$, then $a_{3} = 4a_{1}, a_{4} = 6a_{1}$ and  $a_{5} = 10a_{1}$, which implies that $A = a_{1}\ast \{1,3,4,6,10\}$, and so $\left|4_{\pm}^{\wedge}A \right| \geq 26$.\\
	$\clubsuit$ If $\beta_{1} = x_{2}$ and $z_{4} = x_{5}$, then $a_{3} = 4a_{1}, a_{4} = 6a_{1}$ and  $a_{5} = 12a_{1}$, and so  $A = a_{1}\ast \{1,3,4,6,12\}$, and so $\left|4_{\pm}^{\wedge}A \right| \geq 26$. \\
	
	\noindent \textbf{Case F} 
	$(a_{3} - a_{2}  = a_{2} - a_{1}  = 2a_{1}$, $a_{5} - a_{4}  \neq  2a_{1} $ and  $a_{4} - a_{3} \neq 2a_{1})$. In this case,
	\begin{multicols}{2}
		\begin{itemize}
			\item [] $a_{2} = 3a_{1}$, $a_{3} = 5a_{1} < a_{4} \neq 7a_{1}$,
			\item [] $\gamma_{1} = -a_{1} + a_{2} + a_{3} - a_{4} = 7a_{1} - a_{4}$,
			\item [] $x_{1} = -a_{1} + a_{2} - a_{3} + a_{4} = -3a_{1} + a_{4}$,
			\item [] $x_{2} = a_{1} + a_{2} - a_{3} + a_{4} = -a_{1} + a_{4}$,
			
			\item [] $\beta_{1} = -a_{1} - a_{2} + a_{3} + a_{4} = a_{1} + a_{4}$,
			\item [] $x_{3} = a_{1} -a_{2} + a_{3} + a_{4} = 3a_{1} +  a_{4}$,
			\item [] $x_{4} = -a_{1} + a_{2} + a_{3} + a_{4} = 7a_{1} + a_{4}$,
			\item [] $x_{5} = a_{1} + a_{2} + a_{3} + a_{4} = 9a_{1} +a_{4}$,
			\item [] $z_{5} =  -a_{1} + a_{2} + a_{3} + a_{5} = 7a_{1} + a_{5}$,
			\item [] $x_{6} =  a_{1} + a_{2} + a_{3} + a_{5} = 9a_{1} + a_{5}$,
			\item [] $y_{6} = -a_{1} + a_{2} + a_{4} + a_{5} > 7a_{1} +a_{5}$,
			\item [] $x_{7} = a_{1} + a_{2} + a_{4} + a_{5} $,
			\item [] $\delta_{1} = a_{2} - a_{3} - a_{4} + a_{5}$.
			
		\end{itemize}
	\end{multicols}
	\noindent	Clearly,
	\begin{equation}\label{Lemma 3.4 eqn -6}
		0<\{\gamma_{1} \ \text{or} \ -\gamma_{1}\} < x_{1} < x_{2} < \beta_{1} < x_{3} < x_{4} < \{x_{5}, z_{5}\} < x_{6} < x_{7} < x_{8} < x_{9},
	\end{equation}
	and
	\begin{equation*}
		x_{4} < z_{5} < y_{6} < x_{7} \ \text{and} \ y_{6} \neq x_{6}.
	\end{equation*}
	In $(\ref{Lemma 3.4 eqn -6})$, we have $12$ distinct positive integers of $4_{\pm}^{\wedge}A$. Next, we show the existence of at least one more positive integer in $4_{\pm}^{\wedge}A$ different from the elements listed in $(\ref{Lemma 3.4 eqn -6})$. Consider the following cases:\\
	$\clubsuit$ If $y_{6} \neq x_{5}$, then we have the following list of  $13$ positive integers
	\[0<\{\gamma_{1} \ \text{or} \ -\gamma_{1}\} < x_{1} < x_{2} < \beta_{1} < x_{3} < x_{4} < \{x_{5}, z_{5}, y_{6}, x_{6}\} < x_{7} < x_{8} < x_{9}.\]
	$\clubsuit$ If $y_{6} = x_{5}$ and  $-\delta_{1} \neq \gamma_{1}$, then  $a_{5} = 7a_{1}, \gamma_{1} >0$ and $0 < -\delta_{1} = -a_{2} +a_{3} + a_{4} -a_{5} = -5a_{1} + a_{4} < x_{1}$. Thus we have the following list of  $13$ positive integers
	\[0<\{\gamma_{1}, -\delta_{1}\} < x_{1} < x_{2} < \beta_{1} < x_{3} < x_{4} < \{x_{5}, z_{5} \} < x_{6} < x_{7} < x_{8} < x_{9}.\]
	$\clubsuit$ If $y_{6} = x_{5}$ and  $-\delta_{1} = \gamma_{1}$, then  $a_{5} = 7a_{1}$ and $a_{4} = 6a_{1}$, which implies that $A = a_{1} \ast \{1,3,5,6,7\}$, and so $\left|4_{\pm}^{\wedge}A \right| \geq 26$.
\end{proof}

\begin{lemma}\label{main lemma 3}
	Let $A=\{a_{1}, a_{2}, a_{3}, a_{4}, a_{5}\}$ be a set of positive integers, where $a_{1} < a_{2} < a_{3} < a_{4} < a_{5}$. Let $a_{i} - a_{i - 1} = 2a_{1}$ for exactly one $i \in [2,5]$ and $a_{i} - a_{i-1} \neq 2a_{1}$ for the remaining $i \in [2,5]$.  Then $\left|4_{\pm}^{\wedge}A \right| \geq 26$.
\end{lemma}
\begin{proof}
	It is sufficient to prove that there are $13$ positive integers in $4_{\pm}^{\wedge}A$. Consider the following cases.\\
	\noindent {\bf Case A}  $(a_{5} - a_{4} \neq 2a_{1} $,  $a_{4} - a_{3} \neq 2a_{1}$,  $a_{3} - a_{2} \neq 2a_{1}$ and  $a_{2} - a_{1}  = 2a_{1}$). Consider the following list of  integers which are elements of $4_{\pm}^{\wedge}A$
	\begin{multicols}{2}
		\begin{itemize}
			\item [] $\gamma_{1} = -a_{1} + a_{2} + a_{3} - a_{4} $,
			\item [] $x_{1} = -a_{1} + a_{2} - a_{3} + a_{4}$,
			\item [] $x_{2} = a_{1} + a_{2} - a_{3} + a_{4}$,
			\item [] $\delta_{2} = -a_{1} + a_{2} - a_{3} + a_{5}$,
			\item [] $z_{3} = a_{1} + a_{2} - a_{3} + a_{5}$,
			\item [] $z_{4} = a_{1} - a_{2} + a_{3} + a_{5}$,
			\item [] $z_{5} = -a_{1} + a_{2} + a_{3} + a_{5}$,
			\item [] $x_{6} = a_{1} + a_{2} + a_{3} + a_{5}$,
			\item [] $y_{6} = -a_{1} + a_{2} + a_{4} + a_{5}$,
			\item [] $x_{7} = a_{1} + a_{2} + a_{4} + a_{5}$,
			\item [] $x_{8} = a_{1} + a_{3} + a_{4} + a_{5}$,
			\item [] $x_{9} = a_{2} + a_{3} + a_{4} + a_{5}$,
			\item [] $y_{7} = -a_{1} + a_{3} + a_{4} + a_{5}$,
			\item [] $\gamma_{2} = a_{1} + a_{2} + a_{3} - a_{4}$,
			\item [] $x_{5} = a_{1} + a_{2} + a_{3} + a_{4}$,
			\item [] $y_{5} = a_{1} - a_{2} + a_{4} + a_{5}$,
			\item [] $y_{2} = -a_{1} + a_{2} - a_{4} + a_{5}$.
		\end{itemize}
	\end{multicols}
	\noindent Note that
	\begin{equation}\label{Lemma 3.5 Eqn-1}
		0<\{\gamma_{1} \ \text{or} \ -\gamma_{1}\} < x_{1} < \{x_{2}, \delta_{2}\} < z_{3} < z_{4} < z_{5} < \{x_{6}, y_{6}\} < x_{7} < x_{8} < x_{9}.
	\end{equation} 
	Also
	\[z_{5} < y_{6} < y_{7} < x_{8}, \  y_{7} \neq x_{7} \ \text{and} \ z_{4} < y_{5} < y_{6} < x_{7}.\]
	Therefore, we have at least $12$ distinct positive integers in $4_{\pm}^{\wedge}A$ listed in $(\ref{Lemma 3.5 Eqn-1})$. Next, we show the existence of at least one more positive integer in $4_{\pm}^{\wedge}A$ different from the elements listed in $(\ref{Lemma 3.5 Eqn-1})$. Consider the following cases:\\   
	$\clubsuit$ If $y_{7} \neq x_{6}$, then we have a list of $13$ distinct positive integers of $4_{\pm}^{\wedge}A$
	\[0<\{\gamma_{1} \ \text{or} \ -\gamma_{1}\} < x_{1} < \{x_{2}, \delta_{2}\} < z_{3} < z_{4} < z_{5} < \{x_{6}, y_{6}, x_{7}, y_{7}\} < x_{8} < x_{9}.\]
	$\clubsuit$ If $y_{7} = x_{6}$, then $a_{4} = 5a_{1}$ and $y_{5} <  z_{5}$. Therefore, we have
	\[0<\{\gamma_{1} \ \text{or} \ -\gamma_{1}\} < x_{1} < \{x_{2}, \delta_{2}\} < z_{3} < z_{4} <  y_{5} < z_{5} < \{x_{6}, y_{6}\} < x_{7} < x_{8} < x_{9}.\]
	Thus, we have at least 26 elements in $4_{\pm}^{\wedge}A$ in this case.\\
	
	\noindent {\bf Case B}  
	$(a_{5} - a_{4}  \neq 2a_{1}$, $a_{4} - a_{3}\neq2a_{1}$, $a_{3} - a_{2} =  2a_{1} $ and  $a_{2} - a_{1}  \neq  2a_{1})$.
	Consider the following list of integers which are elements of $4_{\pm}^{\wedge}A$
	\begin{multicols}{2}
		\begin{itemize}
			\item [] $x_{1} = -a_{1} + a_{2} - a_{3} + a_{4}$,
			\item [] $x_{2} = a_{1} + a_{2} - a_{3} + a_{4}$,
			\item [] $\beta_{1} = -a_{1} - a_{2} + a_{3} + a_{4}$,
			\item [] $z_{3} = a_{1} + a_{2} - a_{3} + a_{5}$,
			\item [] $\beta_{2} = -a_{1} - a_{2} + a_{3} + a_{5}$,
			\item [] $y_{4} = -a_{1} - a_{2} + a_{4} + a_{5}$,
			\item [] $z_{4} = a_{1} - a_{2} + a_{3} + a_{5}$,
			\item [] $y_{5} = a_{1} - a_{2} + a_{4} + a_{5}$,
			\item [] $y_{6} = -a_{1} + a_{2} + a_{4}  + a_{5}$,
			\item [] $x_{7} = a_{1} + a_{2} + a_{4} + a_{5}$,
			\item [] $x_{8} = a_{1} + a_{3} + a_{4} + a_{5}$,
			\item [] $x_{9} = a_{2} + a_{3} + a_{4} + a_{5}$,
			\item [] $\delta_{2} = -a_{1} + a_{2} - a_{3} + a_{5}$,
			\item [] $\delta_{1} = a_{2} - a_{3} - a_{4} + a_{5}$,
			\item [] $y_{2} = -a_{1} + a_{2} - a_{4} + a_{5}$.
		\end{itemize}
	\end{multicols}
	\noindent	Note that
	\begin{equation}\label{Lemma 3.5 Eqn-2}
		0 < x_{1} < x_{2} < \{\beta_{1}, z_{3}\} < \beta_{2} < \{y_{4}, z_{4}\} < y_{5} < y_{6} < x_{7} < x_{8} < x_{9}.
	\end{equation}
	Also
	\[x_{1} < \delta_{2} < z_{3} < \beta_{2}, ~ \delta_{1} < y_{2} < \delta_{2} < z_{3} < \beta_{2} \ \text{and} \ \delta_{2} \neq x_{2}.\]
	In $(\ref{Lemma 3.5 Eqn-2})$, we have $12$ distinct positive integers of $4_{\pm}^{\wedge}A$. Next, we show the existence of at least one more positive integer in $4_{\pm}^{\wedge}A$ different from the elements listed in $(\ref{Lemma 3.5 Eqn-2})$. Consider the following cases:\\
	$\clubsuit$ If $\delta_{2} \neq \beta_{1}$, then we have at least $13$ positive integers in $4_{\pm}^{\wedge}A$, listed below
	\[0 < x_{1} < \{x_{2}, \beta_{1}, z_{3}, \delta_{2}\} < \beta_{2} < \{y_{4}, z_{4}\} < y_{5} < y_{6} < x_{7} < x_{8} < x_{9}.\]
	$\clubsuit$ If $\delta_{2} = \beta_{1}$ and $\delta_{1} \neq x_{1}$, then $a_{5} - a_{4} = 2a_{3} - 2a_{2} = 4a_{1}$ and $\delta_{1} = 2a_{1} < x_{2}$. Therefore, we have at least $13$ positive integers in  $4_{\pm}^{\wedge}A$, listed below
	\[0 < \{x_{1}, \delta_{1}\} < x_{2} < \delta_{2}< z_{3} < \beta_{2} < \{y_{4}, z_{4}\} < y_{5} < y_{6} < x_{7} < x_{8} < x_{9}.\]
	$\clubsuit$ If $\delta_{2} = \beta_{1}$, and  $\delta_{1} = x_{1}$, then  $a_{4} =  5a_{1}$ and $a_{5} = 9a_{1}$. Consequently $x_{2} < y_{2} < z_{3}$, and so
	\[\delta_{1} = x_{1} < x_{2} < y_{2} < \delta_{2} = \beta_{1} < z_{3} <\beta_{2}  < \{y_{4}, z_{4}\} < y_{5} < y_{6} < x_{7} <  x_{8} < x_{9} .\]
	Thus, we have at least $26$ elements in $4_{\pm}^{\wedge}A$ in this case also.\\
	
	\noindent {\bf Case C} $(a_{5} - a_{4} \neq 2a_{1} $, $a_{4} - a_{3} = 2a_{1}$,  $a_{3} - a_{2} \neq 2a_{1}$ and  $a_{2} - a_{1} \neq 2a_{1})$. Consider the following list of positive integers which are elements of $4_{\pm}^{\wedge}A$
	\begin{multicols}{2}
		\begin{itemize}
			\item [] $\gamma_{1} = -a_{1} + a_{2} + a_{3} - a_{4}$,
			
			\item [] $x_{1} = -a_{1} + a_{2} - a_{3} + a_{4}$,
			\item [] $x_{2} = a_{1} + a_{2} - a_{3} + a_{4}$,
			\item [] $\beta_{1} = -a_{1} - a_{2} + a_{3} + a_{4}$,
			\item [] $x_{3} = a_{1} - a_{2} + a_{3} + a_{4}$,
			\item [] $x_{4} = -a_{1} + a_{2} + a_{3} + a_{4}$,
			\item [] $x_{5} = a_{1} + a_{2} + a_{3} + a_{4}$,
			\item [] $z_{5} = -a_{1} + a_{2} + a_{3} + a_{5}$,
			\item [] $x_{6} = a_{1} + a_{2} + a_{3} + a_{5}$,
			\item [] $x_{7} = a_{1} + a_{2} + a_{4} + a_{5}$,
			\item [] $y_{7} = -a_{1} + a_{3} + a_{4} + a_{5}$,
			\item [] $x_{8} = a_{1} + a_{3} + a_{4} + a_{5}$,
			\item [] $x_{9} = a_{2} + a_{3} + a_{4} + a_{5}$.
		\end{itemize}
	\end{multicols}
	\noindent	Clearly, $$x_{1}<\beta_{1}<x_{3}, \ x_{4} < z_{5} < x_{6}, \ x_{5} \neq z_{5}, \ x_{7} \neq y_{7}, \ y_{7} < x_{8}.$$ Since $a_{4} -a_{2} > a_{4} - a_{3} = 2a_{1}$, we have $x_{6} < y_{7}$. Therefore from (\ref{Inequalities}), we have at least $12$ distinct positive integers  in $4_{\pm}^{\wedge}A$, listed below
	\begin{equation}\label{Lemma 3.5 Eqn-3}
		0<\{\gamma_{1} \ \text{or} \ -\gamma_{1} \} < x_{1} < x_{2} < x_{3} < x_{4} < \{x_{5}, z_{5} \} < x_{6} < \{x_{7}, y_{7} \} < x_{8} < x_{9}.
	\end{equation}
	Next, we show the existence of at least one more positive integer in $4_{\pm}^{\wedge}A$ different from the elements listed in $(\ref{Lemma 3.5 Eqn-3})$. Consider the following cases:\\
	$\clubsuit$ If $x_{2} \neq \beta_{1}$, then we have at least $13$ positive integers, listed below
	\[0<\{\gamma_{1} \ \text{or} \ -\gamma_{1} \} < x_{1} < \{ x_{2}, \beta_{1} \} < x_{3} < x_{4} < \{x_{5}, z_{5} \} < x_{6} < \{x_{7}, y_{7} \} < x_{8} < x_{9}.\]
	$\clubsuit$ If $x_{2} = \beta_{1}$, then $a_{3} - a_{2} = a_{1}$. Therefore, $a_{5} > a_{4} = a_{3} +2a_{1} = a_{2} + 3a_{1} > 4a_{1} $.  Consider the following increasing sequence of  elements of  $4_{\pm}^{\wedge}A$
	\begin{multline*}
		0 < -4a_{1} + a_{5} = -a_{1} + a_{2} - a_{4} + a_{5} < -a_{1} + a_{3} - a_{4} + a_{5} < -a_{1} + a_{2} - a_{3} + a_{5} \\< a_{1} + a_{3} - a_{4} + a_{5} < -a_{1} - a_{2} + a_{3} + a_{5} < -a_{1} - a_{3} + a_{4} + a_{5} < a_{1} - a_{2} + a_{3} + a_{5} \\ < a_{1} - a_{3} + a_{4} + a_{5} < a_{1} - a_{2} + a_{4} + a_{5} < -a_{1} + a_{2} + a_{4} + a_{5} < a_{1} +a_{2} +a_{4} + a_{5} \\< a_{1} + a_{3} + a_{4} + a_{5} < a_{2} + a_{3} + a_{4} + a_{5}.
	\end{multline*}
	Thus, we have at least 26 elements in $4_{\pm}^{\wedge}A$ in this case also.\\	
	
	\noindent {\bf Case D}  $(a_{5} - a_{4} = 2a_{1} $, $a_{4} - a_{3} \neq 2a_{1}$,  $a_{3} - a_{2} \neq 2a_{1}$ and  $a_{2} - a_{1} \neq 2a_{1})$.  	Consider the following list of positive integers which are elements of $4_{\pm}^{\wedge}A$
	
	\begin{multicols}{2}
		\begin{itemize}
			\item [] $z_{1} = a_{1} + a_{2} + a_{4} - a_{5}$,
			\item [] $x_{1} = -a_{1} + a_{2} - a_{3} + a_{4}$,
			\item [] $y_{2} = -a_{1} + a_{2} - a_{4} + a_{5}$,
			\item [] $\delta_{2} = -a_{1} + a_{2} - a_{3} + a_{5}$,
			\item [] $\beta_{2} = -a_{1} - a_{2} + a_{3} + a_{5}$,
			\item [] $y_{4} = -a_{1} - a_{2} + a_{4} + a_{5}$,
			\item [] $z_{4} = a_{1} - a_{2} + a_{3} + a_{5}$,
			\item []  $y_{5} = a_{1} - a_{2} + a_{4} +  a_{5}$,
			\item [] $y_{6} = -a_{1} + a_{2} + a_{4} + a_{5}$,
			\item  [] $ x_{7} = a_{1} + a_{2} + a_{4} + a_{5}$,
			\item [] $y_{7} = -a_{1} + a_{3} + a_{4} + a_{5}$,
			\item [] $x_{8} = a_{1} + a_{3} + a_{4} + a_{5}$,
			\item [] $x_{9} = a_{2} + a_{3} + a_{4} + a_{5}$.
		\end{itemize}
	\end{multicols}
	\noindent Note that $0<z_{1} < \{x_{1}, y_{2}\} < \delta_{2} < \beta_{2} < \{y_{4}, z_{4}\} < y_{5} < y_{6} < \{x_{7}, y_{7}\} < x_{8} < x_{9}$.
	Thus, we have at least 26 elements in $4_{\pm}^{\wedge}A$ in this case. 	
\end{proof}

\begin{lemma}\label{main lemma 4(a)}
	Let $A = \{a_{1}, a_{2}, a_{3}, a_{4}, a_{5}\}$ be a set of positive integers, where $a_{1} < a_{2} < a_{3} < a_{4} < a_{5}$ with $a_{i}  - a_{i-1} \neq 2a_{1}$ for all $i \in [2,5]$ and $a_{4} - a_{3} \neq a_{2} - a_{1}$. Then $\left|4_{\pm}^{\wedge}A \right| \geq 26$.
\end{lemma}
\begin{proof} 	 Consider the following list of  integers (not necessarily distinct)
	\begin{multicols}{2}
		\begin{itemize}
			\item [] $\gamma_{1} = -a_{1} + a_{2} + a_{3} - a_{4}$,
			
			\item [] $x_{1} = -a_{1} + a_{2} - a_{3} + a_{4}$,
			\item [] $x_{2}  = a_{1} + a_{2} - a_{3} + a_{4}$,
			\item  [] $\beta_{1} = -a_{1} - a_{2} + a_{3} + a_{4}$,
			\item [] $x_{3} = a_{1} - a_{2} + a_{3} + a_{4}$,
			\item [] $x_{4} = -a_{1} + a_{2} + a_{3} + a_{4}$,
			\item [] $x_{5} = a_{1} + a_{2} + a_{3} + a_{4}$,
			\item [] $z_{5} = -a_{1} + a_{2} + a_{3} + a_{5}$,
			\item [] $x_{6} = a_{1} + a_{2} + a_{3} + a_{5}$,
			\item [] $x_{7} = a_{1} + a_{2} + a_{4} + a_{5}$,
			\item [] $x_{8} = a_{1} + a_{3} + a_{4} + a_{5}$,
			\item [] $x_{9} = a_{2} + a_{3} + a_{4} + a_{5}$,
			\item [] $y_{5} = a_{1}-a_{2}+a_{4}+a_{5}$,
			\item [] $y_{6} = -a_{1} + a_{2} + a_{4} + a_{5}$,
			\item [] $y_{7} = -a_{1} + a_{3} + a_{4} + a_{5}$,
			\item [] $\epsilon_{1} = -a_{1} + a_{3} + a_{4} - a_{5}$,
			\item  [] $\delta_{2} = -a_{1} + a_{2} - a_{3} + a_{5}$,
			\item [] $z_{3} = a_{1}+a_{2}-a_{3}+a_{5}$.
		\end{itemize}
	\end{multicols}
	\noindent Since  $a_{i} -a_{i-1} \neq 2a_{1}$ for all $i \in [2,5]$, and $a_{4} - a_{3} \neq a_{2} - a_{1}$, we have
	\begin{equation}\label{L4-Eqn-1}
		y_{6} \neq x_{6},~ y_{7} \neq x_{7},~ x_{5} \neq  z_{5},~   x_{2} \neq \delta_{2} \ \text{and} \ \gamma_{1} \neq 0.
	\end{equation}
	Therefore from (\ref{Inequalities})  and (\ref{L4-Eqn-1}) we have,
	\begin{equation*}
		0 < \{\gamma_{1} \ \text{or} \ -\gamma_{1}\} < x_{1} < x_{2} < x_{3} <x_{4} < \{z_{5}, x_{5}\} <x_{6} <x_{7} < x_{8} <x_{9}.
	\end{equation*}
	Also
	\begin{equation*}
		x_{4} < z_{5} < y_{6} < x_{7}, ~x_{4} < z_{5} < y_{7} < x_{8}, ~x_{1} < \beta_{1} < x_{3} \ \text{and} \ x_{1} < \delta_{2}.
	\end{equation*}
	Consider the following cases:
	
	\noindent {\bf Case A}  $\boldsymbol{(a_{3} \neq a_{2} + a_{1})}.$ In this case, $x_{2} = a_{1} + a_{2} -a _{3} + a_{4} \neq -a_{1} - a_{2}  + a_{3} + a_{4} = \beta_{1}$. Now, consider the following cases.\\
	$\clubsuit$ If $y_{6} \neq x_{5}$, then we have at least $13$ positive integers in $4^{\wedge}_{\pm}A$, listed below
	\begin{equation*}
		0 < \{\gamma_{1} \ \text{or} \ -\gamma_{1}\} < x_{1} < \{ x_{2}, \beta_{1} \} < x_{3} <x_{4} < \{z_{5}, x_{5}, y_{6}, x_{6} \} < x_{7} < x_{8} < x_{9}.
	\end{equation*}
	$\clubsuit$	If $y_{6}=x_{5}$ and $y_{7} \neq x_{6}$, then  we have at least $13$ positive integers in $4^{\wedge}_{\pm}A$, listed below
	\begin{equation*}
		0 < \{\gamma_{1} \ \text{or} \ -\gamma_{1}\} < x_{1} < \{ x_{2}, \beta_{1} \} < x_{3} < x_{4} < \{z_{5}, y_{6}, y_{7}, x_{6} \} < x_{7} < x_{8} < x_{9}.
	\end{equation*}
	$\clubsuit$  If $y_{6} = x_{5}$ and $y_{7} = x_{6}$, then  $a_{5} - a_{3} = 2a_{1}$ and $a_{4} - a_{2} = 2a_{1}$. Therefore,
	\[\epsilon_{1} = -a_{1} + a_{3} + a_{4} -a_{5} = -a_{1} + a_{4} -2a_{1} = - a_{1} + a_{2} < x_{1},\]
	\[ x_{1} < \delta_{2} = -a_{1} + a_{2} - a_{3} + a_{5} = a_{1} + a_{2} < a_{1} + a_{2} -a_{3} + a_{4}  = x_{2},\]
	and
	\[\beta_{1} = -a_{1} - a_{2} + a_{3} + a_{4} = a_{1} + a_{3} > a_{1} +a_{2} = \delta_{2}.\]
	It follows that
	\begin{equation*}
		0 < \epsilon_{1} <  x_{1} < \delta_{2} < \{ x_{2}, \beta_{1} \} < x_{3} < x_{4} < \{z_{5}, x_{5}\} < x_{6} < x_{7} < x_{8} < x_{9}.
	\end{equation*}
	Therefore, in each case, we have at least $13$ positive integers in $4^{\wedge}_{\pm}A$.\\

	\noindent {\bf Case B}  $\boldsymbol{(a_{3} = a_{2} + a_{1})}$. From  (\ref{Inequalities}) and (\ref{L4-Eqn-1}), we have
	\begin{equation*}
		0 < \{\gamma_{1} \ \text{or} \ -\gamma_{1}\} < x_{1} < \{ x_{2}, \delta_{2} \} < z_{3} < z_{4} < z_{5} < \{x_{6},  y_{6} \} < x_{7} < x_{8} < x_{9}.
	\end{equation*}
	Also
	\begin{equation*}
		x_{2} < x_{3} <	z_{4} < y_{5} < y_{6} < x_{7}, ~z_{5} < y_{6} < y_{7} < x_{8}  \ \text{and} \ y_{7} \neq x_{7}.
	\end{equation*}
	Observe the following.
	\begin{enumerate}
		\item If $y_{7} = x_{6}$, then $a_{4} - a_{2} = 2a_{1}$. So $y_{5} < x_{6}$.
		
		\item If $y_{7} = x_{6}$ and $y_{5} = z_{5}$, then  $a_{4} -a_{2} = 2a_{1}, a_{4} - a_{3} = a_{1}$ and $a_{4} -a_{3} = 2a_{2} -2a_{1}$. This gives $2a_{2} = 3a_{1}, 2a_{3} = 5a_{1}, 2a_{4} = 7a_{1}$ and $x_{3} \neq z_{3}$.
	\end{enumerate}
	Now, consider the following cases:\\
	$\clubsuit$ If $y_{7} \neq x_{6}$, then we have at least $13$ positive integers in $4^{\wedge}_{\pm}A$, listed below
	\begin{equation*}
		0 < \{\gamma_{1} \ \text{or} \ -\gamma_{1}\} < x_{1} < \{ x_{2}, \delta_{2} \} < z_{3} <z_{4} < z_{5} < \{x_{6},  y_{6}, x_{7}, y_{7} \} < x_{8} < x_{9}.
	\end{equation*}
	$\clubsuit$ If $y_{7} = x_{6}$ and $y_{5} \neq z_{5}$, then we have at least $13$ positive integers in $4^{\wedge}_{\pm}A$, listed below
	\begin{equation*}
		0 < \{\gamma_{1} \ \text{or} \ -\gamma_{1}\} < x_{1} < \{ x_{2}, \delta_{2} \} < z_{3} < z_{4} < \{z_{5}, y_{5},  y_{6}, x_{6}, x_{7} \} < x_{8} < x_{9}.
	\end{equation*}
	$\clubsuit$ If $y_{7} = x_{6}, y_{5} = z_{5}$ and $x_{3} \neq \delta_{2}$, then we have at least $13$ positive integers in $4^{\wedge}_{\pm}A$, listed below
	\begin{equation*}
		0 < \{\gamma_{1} \ \text{or} \ -\gamma_{1}\} < x_{1} < \{ x_{2}, \delta_{2}, z_{3}, x_{3} \} < z_{4} < z_{5} < \{ x_{6}, y_{6} \} < x_{7} < x_{8} < x_{9}.
	\end{equation*}
	$\clubsuit$ If $y_{7} = x_{6}, y_{5} = z_{5}$ and $x_{3} = \delta_{2}$, then  $2a_{2} = 3a_{1}, 2a_{3} = 5a_{1}, 2a_{4} = 7a_{1}$, and  $2a_{5} = 15a_{1}$. In this case, $2 \ast A = a_{1} \ast \{2,3,5,7,15\}$, and so $|4^{\wedge}_{\pm}A| = |4^{\wedge}_{\pm}(2 \ast A)| \geq 26$.\\
	Therefore, in each of these cases we have at least $13$ positive integers in $4^{\wedge}_{\pm}A$.
\end{proof}

\begin{lemma}\label{main lemma 4(b)}
	Let $A = \{a_{1}, a_{2}, a_{3}, a_{4}, a_{5}\}$ be a set of positive integers, where $a_{1} < a_{2} < a_{3} < a_{4} < a_{5}$ with $a_{4} - a_{3} = a_{2} - a_{1}$, $a_{5} - a_{3} \neq 2a_{1}$  and $a_{i}  - a_{i-1} \neq 2a_{1}$ for all $i \in [2,5]$. Then $\left|4_{\pm}^{\wedge}A \right| \geq 26$.
\end{lemma}

\begin{proof}
	Consider the following list of positive integers (not necessarily distinct) of $4_{\pm}^{\wedge}A$
	\begin{multicols}{2}
		\begin{itemize}
			\item [] $\gamma_{1} = -a_{1} + a_{2} + a_{3} - a_{4}$,
			
			\item [] $x_{1} = -a_{1} + a_{2} - a_{3} + a_{4}$,
			\item [] $x_{2}  = a_{1} + a_{2} - a_{3} + a_{4}$,
			\item  [] $\beta_{1} = -a_{1} - a_{2} + a_{3} + a_{4}$,
			\item [] $x_{3} = a_{1} - a_{2} + a_{3} + a_{4}$,
			\item [] $x_{4} = -a_{1} + a_{2} + a_{3} + a_{4}$,
			\item [] $x_{5} = a_{1} + a_{2} + a_{3} + a_{4}$,
			\item [] $z_{5} = -a_{1} + a_{2} + a_{3} + a_{5}$,
			\item [] $x_{6} = a_{1} + a_{2} + a_{3} + a_{5}$,
			\item [] $x_{7} = a_{1} + a_{2} + a_{4} + a_{5}$,
			\item [] $x_{8} = a_{1} + a_{3} + a_{4} + a_{5}$,
			\item [] $x_{9} = a_{2} + a_{3} + a_{4} + a_{5}$,
			\item [] $y_{6} = -a_{1} + a_{2} + a_{4} + a_{5}$,
			\item [] $y_{7} = -a_{1} + a_{3} + a_{4} + a_{5}$,
			\item [] $\gamma_{2} = a_{1} + a_{2} + a_{3} - a_{4}$,
			\item [] $z_{4} = a_{1} - a_{2} + a_{3} + a_{5}$,	
			\item [] $y_{4} = -a_{1} - a_{2} + a_{4} + a_{5}$,
			\item [] $y_{5} = a_{1} - a_{2} + a_{4} + a_{5}$,
			\item [] $z_{3} = a_{1}+a_{2}-a_{3}+a_{5}$.
		\end{itemize}
	\end{multicols}
	\noindent	Since  $a_{5} - a_{3} \neq 2a_{1}$  and $a_{i}  - a_{i-1} \neq 2a_{1}$ for all $i \in [2,5]$, we have
	\begin{equation}
		x_{5}  \neq  y_{6}, ~y_{6} \neq x_{6}, ~y_{7} \neq x_{7}, ~ \text{and} ~x_{5} \neq  z_{5}.
	\end{equation}
	Therefore, from (\ref{Inequalities}) we have
	\begin{equation}\label{Lemma 3.7 eq-1}
		0 < x_{1} < x_{2} < x_{3} < x_{4} < \{x_{5},  z_{5}, x_{6}, y_{6} \} < x_{7} < x_{8} < x_{9}.
	\end{equation}
	
	\noindent {\bf Case A}  $\boldsymbol{(a_{4} \neq a_{2} + a_{1})}.$ In this case, $\gamma_{2}  \neq  \beta_{1}$. Since $0 < \gamma_{2} < x_{2}$ and  $x_{1} < \beta_{1} < x_{3}$, we have the following cases:
	
	\noindent $\clubsuit$ If $\gamma_{2} \neq x_{1}$ and $\beta_{1} \neq x_{2}$, then we can add $\beta_{1}$ and $\gamma_{2}$ in the list (\ref{Lemma 3.7 eq-1}) to get $13$ positive integers in $4_{\pm}^{\wedge}A$, listed below
	\begin{equation*}
		0 < \{x_{1}, x_{2}, \beta_{1}, \gamma_{2}\} < x_{3} < x_{4} < \{x_{5},  z_{5}, x_{6}, y_{6} \} < x_{7} < x_{8} < x_{9}.
	\end{equation*}

	\noindent $\clubsuit$ If $\gamma_{2} = x_{1}$ and $\beta_{1} = x_{2}$, then $a_{4} - a_{3} = a_{1}$ and $a_{3} - a_{2}  = a_{1}$. Since  $a_{4} - a_{3} = a_{2} - a_{1}$, we have $a_{2} = 2a_{1}, a_{3} = 3a_{1}$, $a_{4} = 4a_{1}$. Therefore, $x_{2} = 4a_{1}$, $x_{3} = 6a_{1}$, $x_{4} = 8a_{1}$, $x_{5} = 10a_{1}$, $x_{6}  = 6a_{1} + a_{5}$, $z_{5} = 4a_{1} + a_{5}$, $y_{5} = 3a_{1} + a_{5}$, $y_{4} = a_{1} + a_{5}$, $a_{5} \neq 5a_{1}$, $a_{5} \neq 6a_{1}$. Hence  $x_{2} <  y_{4} <  y_{5} < z_{5} <x_{6}$.   It is easy to verify that $y_{5} \notin \{x_{3}, x_{4}, x_{6}, z_{5}, y_{6}\}$ and $y_{4} \notin \{x_{3}, x_{6}, z_{5}, y_{6}\}$.	Now, consider the following cases:
	\begin{enumerate}
		\item If $y_{4} \neq x_{4}$, $y_{4} \neq x_{5}$ and $y_{5} \neq x_{5}$, then we have
		\begin{equation*}
			0 < x_{1} < x_{2} < \{x_{3},  x_{4}, x_{5},  z_{5}, x_{6},y_{4}, y_{5}, y_{6} \} < x_{7} < x_{8} < x_{9}.
		\end{equation*}
		\item If  $y_{5} = x_{5}$ or  $y_{4} = x_{4}$, then $a_{5} = 7a_{1}$, and so $A = a_{1} \ast \{1,2,3,4,7\}$.
		
		\item If  $y_{4} = x_{5}$, then  $a_{5} = 9a_{1}$, and so  $A = a_{1} \ast \{1,2,3,4,9\}$.
	\end{enumerate}
	In each case, we have at least $13$ positive integers in $4_{\pm}^{\wedge}A$.		
	
	\noindent $\clubsuit$ If $\gamma_{2} = x_{1}$ and $\beta_{1} \neq x_{2}$, then   $a_{4} - a_{3} = a_{1}$ and $a_{3} - a_{2} \neq a_{1}$. Since  $a_{4} - a_{3} = a_{2} - a_{1}$, so $a_{2} = 2a_{1}$ and $a_{3}\neq 3a_{1}$. Consider $y_{5}$ and $y_{7}$. Then we have the following observations.
	
	\begin{enumerate}
		\item If $y_{5} = a_{1} - a_{2} + a_{4} + a_{5} = -a_{1} + a_{2} + a_{3} + a_{4} = x_{4}$, then $a_{5} - a_{3} = 2a_{2} - 2a_{1} = 2(a_{4}-a_{3}) = 2a_{1}$, but $a_{5} - a_{3} \neq 2a_{1}$. Therefore,  $y_{5} \neq x_{4}$.
		
		\item If $y_{5} = a_{1} - a_{2} + a_{4} + a_{5}  = -a_{1} + a_{2} + a_{3} + a_{5} = z_{5}$, then $ a_{4} - a_{3} = 2a_{2} - 2a_{1}$, but $ a_{4}  - a_{3} =  a_{2} - a_{1}$. Therefore, $y_{5} \neq z_{5}$.
		
		\item If $y_{5} = a_{1} - a_{2} + a_{4} + a_{5}  = a_{1} + a_{2} + a_{3} + a_{5} = x_{6}$, then $a_{4} - a_{3} = 2a_{2}$, but $a_{4} - a_{3} = a_{1}$. Therefore, $y_{5} \neq  x_{6}$.
	\end{enumerate}
	\noindent	Since  $x_{3} < y_{5} < y_{6} <  x_{7}$,  $x_{4} < z_{5} < y_{6} < y_{7} < x_{8}$ and $y_{7} \neq x_{7}$, we have the following cases.
	\begin{enumerate}
		\item  If $y_{5}  \neq  x_{5}$, then we have
		\begin{equation*}				
			0 < x_{1} < \{ x_{2}, \beta_{1} \} < x_{3} < \{ x_{4}, x_{5},  z_{5}, x_{6}, y_{5}, y_{6} \} < x_{7} < x_{8} < x_{9}.
		\end{equation*}
		\item 	If $y_{5}  = x_{5}$ and $y_{7} \neq x_{6}$, then we have
		\begin{equation*}
			0 < x_{1} < \{ x_{2}, \beta_{1} \} < x_{3} < \{ x_{4}, x_{5},  z_{5}, x_{6},  y_{6},  y_{7}, x_{7}\} < x_{8} < x_{9}.
		\end{equation*}
		\item If $y_{5}  = x_{5}$ and $y_{7} = x_{6}$, then $a_{5} = a_{3} + 4a_{1}$ and $a_{4} = a_{2} + 2a_{1} = 4a_{1}$. This gives $A = a_{1} \ast \{1,2,3,4,7\}$. 
	\end{enumerate}
	In each case, we have at least $13$ positive integers in $4_{\pm}^{\wedge}A$.	
	
	\noindent $\clubsuit$ If $\gamma_{2} \neq x_{1}$ and $\beta_{1} = x_{2}$, then $a_{3} = a_{2} + a_{1}$, and so $a_{4} = a_{3} + a_{2} - a_{1} = 2a_{2}$. Consider $z_{3}$ and $y_{7}$. Then  we have
	\begin{equation*}
		\left.
		\begin{array}{ll}
			x_{2} < z_{3} < x_{6} <  x_{7}, & \ \ \ \ \ \ \   z_{3} < z_{4} < y_{6} < y_{7} < x_{8}, \\
			x_{4} < z_{5} < y_{7} < x_{8}, &  \ \ \ \ \ \ \ y_{7} \neq x_{7}.
		\end{array}
		\right.
	\end{equation*}
	Now, consider the following cases.
	\begin{enumerate}
		\item If $y_{7} \neq x_{5}$ and $y_{7} \neq x_{6}$, then we have
		\begin{equation*}
			0 < \{\gamma_{2}, x_{1} \} < x_{2} < x_{3} <  x_{4} < \{ x_{5},  z_{5}, x_{6}, y_{6}, y_{7}, x_{7} \} < x_{8} < x_{9}.
		\end{equation*}
		
		\item If $y_{7} = x_{5}$, then $a_{5} = 2a_{1} + a_{2}$, and it follows that $z_{3}  = a_{1} + a_{2} - a_{3} + a_{5} = 2a_{1} + a_{2} < 2a_{2} + 2a_{1} =  a_{1} -a_{2} + a_{3} +a_{4} = x_{3}$. Therefore, we have
		\begin{equation*}
			0 < \{\gamma_{2}, x_{1} \} < x_{2} < z_{3} < x_{3} <  x_{4} < \{ x_{5},  z_{5}, x_{6}, y_{6}\} <  x_{7} < x_{8} < x_{9}.
		\end{equation*}
		
		\item If $y_{7} = x_{6}$, then $ a_{4} = a_{2} + 2a_{1}$. Since $a_{4} = 2a_{2}$, we have $a_{2} = 2a_{1}$, $a_{3} = a_{1} + a_{2} = 3a_{1}$ and $a_{4} = 2a_{2} = 4a_{1}$.  Clearly, $z_{3} \notin  \{z_{5}, y_{6}, x_{7}\}$ and $x_{2} < z_{3} < x_{6}$. Now, consider the following cases:
		\begin{enumerate}
			\item [\upshape(a)] If $z_{3} \notin \{x_{3}, x_{4}, x_{5}\}$, then we have
			\begin{equation*}
				0 < \{\gamma_{2}, x_{1} \} < x_{2} < \{ z_{3}, x_{3},   x_{4},  x_{5},  z_{5}, x_{6}, y_{6}\} <  x_{7} < x_{8} < x_{9}.
			\end{equation*}
			\item [\upshape(b)] If $z_{3} = x_{3}$, then $a_{5} = a_{4} + 2a_{3} - 2a_{2} = 6a_{1}$, and so $A = a_{1} \ast \{1,2,3,4,6\}$.
			\item [\upshape(c)] If $z_{3} = x_{4}$, then $a_{5} = a_{4} + 2a_{3} - 2a_{1} = 8a_{1}$, and so $A = a_{1}\ast\{1,2,3,4,8\}$.
			\item [\upshape(d)] If $z_{3} = x_{5}$, then $a_{5} = a_{4} + 2a_{3} = 10a_{1}$, and so $A = a_{1} \ast \{1,2,3,4,10\}$. 
		\end{enumerate}
		In each case, we have $\left|4_{\pm}^{\wedge}A \right| \geq 26$.			
	\end{enumerate}
	Therefore, in  the case $a_{4} \neq a_{2} + a_{1}$, we  have at least $13$ positive integers in $4^{\wedge}_{\pm}A$.\\
	
	\noindent {\bf Case B}  $\boldsymbol{(a_{4} = a_{2} + a_{1})}.$ In this case, $a_{3} = 2a_{1}$ and $a_{2} < 2a_{1}$. Therefore, $$	x_{1} = -a_{1} + a_{2} - a_{3} + a_{4} = 2a_{2} - 2a_{1} < 2a_{1} =   a_{1} + a_{2} + a_{3} -a_{4}  = \gamma_{2} \ \text{and} \ y_{6} < x_{6}.$$ It follows that
	\begin{equation}\label{Lemma 3.7 eq-2}
		0 < x_{1} < \gamma_{2} <  x_{2} < x_{3} < x_{4} < \{x_{5},  z_{5},  y_{6} \} < x_{6} < x_{7} < x_{8} < x_{9}.
	\end{equation}
	In $(\ref{Lemma 3.7 eq-2})$, we have $12$ distinct positive integers of $4_{\pm}^{\wedge}A$. Next, we show the existence of at least one more positive integer in $4_{\pm}^{\wedge}A$ different from the elements listed in $(\ref{Lemma 3.7 eq-2})$.	Note that $y_{7} = -a_{1} + a_{3} + a_{4} + a_{5} = a_{2} + a_{3} + a_{5} < a_{1} + a_{2} + a_{3} + a_{5} = x_{6}$. Since $x_{4} < z_{5} < y_{6} < y_{7} $, we have the following cases:\\
	$\clubsuit$ If $y_{7} \neq  x_{5} $, then we have  
	\[0 < x_{1} < \gamma_{2} <  x_{2} < x_{3} <  x_{4} < \{x_{5},  z_{5},  y_{6}, y_{7} \} < x_{6} < x_{7} < x_{8} < x_{9}.\]
	$\clubsuit$ If $y_{7} = x_{5} $ and $z_{4} \neq x_{4}$, then $a_{5} -a_{2} = 2a_{1}$. Therefore $z_{4} = 5a_{1} < 3a_{1} + a_{2} + a_{4} = x_{5}$. Since $x_{3} < z_{4}$, we have
	\[0 < x_{1} < \gamma_{2} <  x_{2} < x_{3} < \{ z_{4}, x_{4}, x_{5},  z_{5},  y_{6} \} < x_{6} < x_{7} < x_{8} < x_{9}.\]
	$\clubsuit$ If $y_{7} = x_{5} $ and $z_{4} = x_{4}$, then $a_{5} -a_{2} = 2a_{1}$ and $2a_{2} = 3a_{1} $. This gives  $2 \ast A = \{2,3,4,5,7\}$.  In this case also, $\left|4_{\pm}^{\wedge}A \right| = \left|4_{\pm}^{\wedge}(2 \ast A) \right| \geq 26$.
	Therefore, in  the case $a_{4} = a_{2} + a_{1}$ also, we  have at least $13$ positive integers in $4^{\wedge}_{\pm}A$.	
\end{proof}

\begin{lemma}\label{main lemma 4(c)}
	Let $A = \{a_{1}, a_{2}, a_{3}, a_{4}, a_{5}\}$ be a set of positive integers, where $a_{1} < a_{2} < a_{3} < a_{4} < a_{5}$ with $a_{4} - a_{3} = a_{2} - a_{1}$, $a_{5} - a_{3} = 2a_{1}$  and $a_{i}  - a_{i-1} \neq 2a_{1}$ for all $i \in [2,5]$. Then $\left|4_{\pm}^{\wedge}A \right| \geq 26$.
\end{lemma}
\begin{proof}
	Consider the following list of nonnegative  integers
	\begin{multicols}{2}
		\begin{itemize}
			
			\item [] $x_{1} =  -a_{1} + a_{2} - a_{3} + a_{4}  = 2a_{2} - 2a_{1}$,
			\item [] $y_{2} =  -a_{1} + a_{2} - a_{4} + a_{5} = 2a_{1}$,
			\item [] $\delta_{2} = -a_{1} + a_{2} - a_{3} + a_{5} = a_{1}+a_{2}$,
			\item [] $x_{2} = a_{1} + a_{2} - a_{3}  + a_{4} = 2a_{2}$,
			\item [] $x_{3} = a_{1} - a_{2} + a_{3} + a_{4} = 2a_{3}$,
			\item [] $x_{4} = -a_{1} + a_{2} + a_{3} + a_{4} = 2a_{4}$,
			\item [] $z_{5} = -a_{1} + a_{2} + a_{3} + a_{5} = a_{4} + a_{5}$,
			\item [] $x_{5} = a_{1} + a_{2} + a_{3} + a_{4}$,
			\item [] $x_{6} = a_{1} + a_{2} + a_{3} + a_{5}$,
			\item [] $x_{7} = a_{1} + a_{2} + a_{4} + a_{5}$,
			\item [] $x_{8} =  a_{1} + a_{3} + a_{4} + a_{5}$,
			\item [] $x_{9} = a_{2} + a_{3} + a_{4} + a_{5}$,
			\item [] $z_{4} = a_{1} - a_{2} + a_{3} + a_{5}$,
			\item [] $y_{7} = -a_{1} + a_{3} + a_{4} + a_{5}$.
		\end{itemize}
	\end{multicols}
	\noindent	Clearly, \[0 < y_{2} < \delta_{2} < x_{2} < x_{3} < x_{4} < \{z_{5}, x_{5}\} < x_{6} < x_{7} < x_{8} < x_{9}.\]
	Since $a_{5} - a_{3} = a_{5} - a_{4} + a_{4} - a_{3} = 2a_{1}$, we have $a_{5} - a_{4} = 2a_{1} - a_{4} + a_{3} = 2a_{1} - a_{2} + a_{1} = 3a_{1} - a_{2} \ \text{and} \ a_{5} - a_{4} < 2a_{1}$. Therefore, $z_{5} < x_{5}$. Consider the following cases.\\
	
	\noindent {\bf Case A}  $\boldsymbol{(a_{2} \neq 2a_{1})}$. In this case,  $x_{1} \neq y_{2}$. Therefore, we have
	\[0  < \{x_{1}, y_{2}\} < \delta_{2} < x_{2} < x_{3} < x_{4} < z_{5} < x_{5} < x_{6} < x_{7} < x_{8} < x_{9}.\]
	Also $$x_{3} < z_{4} < z_{5}  < x_{5} < y_{7} < x_{8} \ \text{and} \   y_{7} \neq x_{7}.$$ 
	
	\noindent $\clubsuit$ If $ z_{4} \neq x_{4}$, then we have at least $13$ positive integers in $4_{\pm}^{\wedge}A$, listed below
	\[0  < \{x_{1}, y_{2}\} < \delta_{2} < x_{2} < x_{3} < \{x_{4}, z_{4} \} < z_{5} < x_{5} < x_{6} < x_{7} < x_{8} < x_{9}.\]
	
	\noindent $\clubsuit$ If  $y_{7} \neq x_{6}$, we have at least $13$ positive integers in $4_{\pm}^{\wedge}A$, listed below
	\[0  < \{x_{1}, y_{2}\} < \delta_{2} < x_{2} < x_{3} < x_{4} < z_{5} < x_{5} < \{ x_{6}, x_{7}, y_{7} \} < x_{8} < x_{9}.\]
	
	\noindent $\clubsuit$ If $z_4 = x_4$ and $y_{7} = x_{6}$, then  $a_{5} - a_{4} = 2a_{2} - 2a_{1}$ and $a_{4} - a_{2} = 2a_{1}$. This gives   $3a_{2} = 5a_{1}$, $a_{3} = 3a_{1}$, $a_{5} = 5a_{1}$ and $3a_{4} = 11a_{1}$. Therefore, we have $3 \ast A = a_{1} \ast \{3,5,9,11,15\}$, which implies that $\left|4_{\pm}^{\wedge}A \right| = \left|4_{\pm}^{\wedge}(3\ast A) \right| \geq 26$.\\
	
	\noindent {\bf Case B}  $\boldsymbol{(a_{2} = 2a_{1})}$. In this case, $a_{5} - a_{4} = 3a_{1} - a_{2} = a_{1}$, and $a_{4} - a_{3} = a_{2} - a_{1} = 2a_{1} - a_{1} = a_{1}$. Consider the following list of nonnegative integers
	\begin{multicols}{2}
		\begin{itemize}
			
			\item [] $x_{1} = -a_{1} + a_{2} - a_{3} + a_{4} = 2a_{1}$,
			\item [] $\delta_{2} = -a_{1} + a_{2} - a_{3} + a_{5} = 3a_{1}$,
			\item [] $x_{2} = a_{1} + a_{2} - a_{3} + a_{4}  = 4a_{1}$,
			\item [] $z_{3} =  a_{1} + a_{2} - a_{3} + a_{5} = 5a_{1}$,
			\item [] $z_{4} =  a_{1} - a_{2} + a_{3} + a_{5} = a_{1} + 2a_{3} $,
			\item [] $x_{4} =  -a_{1} + a_{2} + a_{3} + a_{4} = a_{1} + a_{3} + a_{4}$,
			\item [] $z_{5} = -a_{1} + a_{2} + a_{3} + a_{5} = a_{2} + a_{3} + a_{4}$,
			\item [] $x_{5} = a_{1} + a_{2} + a_{3} + a_{4}$,
			\item [] $x_{3} = a_{1} - a_{2} + a_{3} + a_{4} = 2a_{3}$.
		\end{itemize}
	\end{multicols}
	\noindent Clearly,
	\[0  < x_{1} < \delta_{2} < x_{2} < z_{3} < z_{4} < x_{4} < z_{5} < x_{5} < x_{6} < x_{7} < x_{8} < x_{9} \ \text{and} \ x_{2} < x_{3} < z_{4}.\]
	Now, consider the following cases:\\
	\noindent $\clubsuit$ If $x_{3} \neq z_{3}$, then we have at least $13$ positive integers in $4_{\pm}^{\wedge}A$, listed below
	\[0  < x_{1} < \delta_{2} < x_{2} < \{z_{3}, x_{3} \} < z_{4} < x_{4} < z_{5} < x_{5} < x_{6} < x_{7} < x_{8} < x_{9}. \]
	$\clubsuit$ If $x_{3} =  z_{3}$, then  $2a_{3} = a_{1} + 2a_{2} = 5a_{1}$, which gives $2 \ast A = a_{1} \ast \{2,4,5,7,9\}$.  In this case also, $\left|4_{\pm}^{\wedge}A \right| = \left|4_{\pm}^{\wedge}(2 \ast A) \right| \geq 26$. This completes the proof of the lemma.
\end{proof}
Combining Lemma \ref{main lemma 4(a)}, Lemma \ref{main lemma 4(b)} and Lemma \ref{main lemma 4(c)}, we have the following lemma.

\begin{lemma}\label{main lemma 4}
	Let $A = \{a_{1}, a_{2}, a_{3}, a_{4}, a_{5}\}$ be a set of positive integers, where $a_{1} < a_{2} < a_{3} < a_{4} < a_{5}$ with $a_{i}  - a_{i-1} \neq 2a_{1}$ for all $i \in [2,5]$. Then $\left|4_{\pm}^{\wedge}A \right| \geq 26$.
\end{lemma}

Now, we give the proof of Theorem \ref{main theorem}.

\begin{proof}[Proof of Theorem \ref{main theorem}]
	Let $A=\{ a_{1}, a_{2}, a_{3}, a_{4}, a_{5}\}$ be a set of positive integers, where $0 < a_{1} < a_{2} < a_{3} < a_{4} < a_{5}$. If $a_{i}  - a_{i-1} = 2a_{1}$ for all $i \in [2,5]$, then Theorem \ref{thm 1} implies that $\left|4_{\pm}^{\wedge}A \right| = 25$. If $a_{i}  - a_{i-1} \neq 2a_{1}$ for some $i \in [2,5]$, then Lemma \ref{main lemma 1}, Lemma \ref{main lemma 2}, Lemma \ref{main lemma 3} and Lemma \ref{main lemma 4} imply that $\left|4_{\pm}^{\wedge}A \right| \geq 26$.
	
	Conversely, if $\left|4_{\pm}^{\wedge}A \right| = 25$, then $a_{i}  - a_{i-1} = 2a_{1}$ for all $i \in [2,5]$, otherwise Lemma \ref{main lemma 1}, Lemma \ref{main lemma 2}, Lemma \ref{main lemma 3} and Lemma \ref{main lemma 4} imply that $\left|4_{\pm}^{\wedge}A \right| \geq 26$. Thus, $A = a_1 \ast \{1, 3, 5, 7, 9\}$. This completes the proof of the theorem.
\end{proof}
A similar argument as in Theorem \ref{main theorem I}, we have verified and proved the following theorem.
\begin{theorem}\label{main theorem II}
	Let $k \geq 5$ be a positive integer and  $A$ be set of $k$ nonnegative integers with $0 \in A$. Then $$\left|4_{\pm}^{\wedge}A \right| \geq 8k-19.$$
	Furthermore, if  $\left|4_{\pm}^{\wedge}A \right| = 8k-19$, then $A = d \ast [0,k-1]$.
\end{theorem}

\section*{Acknowledgment}
The first author would like to thank to the Council of Scientific and Industrial Research (CSIR), India for providing the grant to carry out the research with Grant No. 09/143(0925)/2018-EMR-I.

\bibliographystyle{amsplain}

\end{document}